\documentclass{amsart}
\allowdisplaybreaks
\usepackage{amsmath,amsfonts,amssymb,bm}%
\usepackage{exscale}%
\usepackage{amsfonts}%
\usepackage{palatino,dsfont}
\usepackage{amsthm}
\usepackage{epsfig}
\usepackage{tikz,xcolor}
\usepackage{hyperref}
\usepackage[utf8]{inputenc}
\usepackage[english]{babel}
\usepackage{mathtools}
\usepackage{cite}
\usepackage[a4paper]{geometry}

\usepackage{cleveref}
\usepackage[footnote,draft,silent,nomargin]{fixme}

\hypersetup{%
	plainpages=false,%
	pdfauthor={Author(s)},%
	pdftitle={Title},%
	pdfsubject={Subject},%
	bookmarksnumbered=true,%
	colorlinks,%
	citecolor=black,%
	filecolor=black,%
	linkcolor= black,
	urlcolor=black,%
	pdfstartview=FitH%
}

\let\epsilon\varepsilon

\let\phi\varphi

\let\theta\vartheta

\newtheorem{mytheorem}{Theorem}[section]
\newtheorem{myprop}[mytheorem]{Proposition}
\newtheorem{mycor}[mytheorem]{Corollary} 
\theoremstyle{definition}
\newtheorem{mydef}[mytheorem]{Definition}

\newtheorem{myre}[mytheorem]{Remark}
\newtheorem{mylemma}[mytheorem]{Lemma}

\newcommand{\C}{\mathbb{C}}
\newcommand{\R}{\mathbb{R}}
\newcommand{\TS}{\textstyle}

\newcommand{\N}{\mathbb{N}}

\newcommand{\cS}{\mathcal{S}}
\newcommand{\cM}{\mathcal{M}}
\newcommand{\cK}{\mathcal{K}}
\newcommand{\cW}{\mathcal{W}}

\newcommand{\va}{\vec{a}}
\newcommand{\vp}{\vec{p}}
\newcommand{\cb}{\mathbf{b}}
\newcommand{\ce}{\mathbf{e}}
\newcommand{\cn}{\mathbf{n}}
\newcommand{\ca}{\mathbf{a}}
\newcommand{\cs}{\mathbf{s}}
\newcommand{\cc}{\mathbf{c}}
\newcommand{\cu}{\mathbf{u}}
\newcommand{\cv}{\mathbf{v}}
\newcommand{\cy}{\mathbf{y}}
\newcommand{\dx}{\,\mathrm{d}\mathbf{x}}
\newcommand{\dy}{\,\mathrm{d}\mathbf{y}}
\newcommand{\cw}{\mathbf{w}}
\newcommand{\cx}{\mathbf{x}}
\newcommand{\cz}{\mathbf{z}}
\newcommand{\Span}{\operatorname{span}}
 \newcommand{\bN}{\mathbb{N}}

\newcommand{\norm}[1]{\left\lVert #1\right\rVert}

\newcommand{\bR}{\mathbb{R}}
\newcommand{\bZ}{\mathbb{Z}}

\newcommand{\supp}{\operatorname{supp}}
\newcommand{\qp}{\mathcal{Q}}
\newcommand{\Dict}{\mathcal{D}}

\newcommand{\K}{\mathcal{K}}

\newcommand{\ptt}{\mathbb{I}}
\newcommand{\brac}[1]{\langle #1\rangle_{\vec{a}}}

\newcommand{\ve}{\varepsilon}
\renewcommand{\epsilon}{\varepsilon}

\numberwithin{equation}{section}

\begin{document}
\title[Painless Construction of Unconditional Bases]{Painless Construction of Unconditional Bases for Anisotropic Modulation and Triebel-Lizorkin Type Spaces }  
\author{Morten Nielsen }  
\address{Department of Mathematical Sciences\\ Aalborg
  University\\ Thomas Manns Vej 23\\ DK-9220 Aalborg East\\ Denmark}
\email{mnielsen@math.aau.dk}
\begin{abstract}
We construct smooth localized orthonormal bases compatible with anisotropic  Triebel-Lizorkin and Besov type spaces on $\bR^d$. The construction is based on tensor products of so-called univariate brushlet functions that are based on local trigonometric bases in the frequency domain, and the construction is painless in the sense that all parameters  for the construction are explicitly specified. It is shown that the associated decomposition system form unconditional bases for the full family of Triebel–Lizorkin and Besov type spaces, including for the so-called $\alpha$-modulation and $\alpha$-Triebel-Lizorkin  spaces. In the second part of the paper we study nonlinear $m$-term approximation with the constructed bases, where direct Jackson and Bernstein inequalities for $m$-term approximation with the tensor brushlet system in $\alpha$-modulation and $\alpha$-Triebel-Lizorkin  spaces are derived. The inverse Bernstein estimates rely heavily on the fact that the constructed system is non-redundant.

\end{abstract}
\maketitle

\section{Introduction}
Well-localized unconditional bases for function spaces defined on $\bR^d$ play a cental role for many applications as the bases offer a framework for discretisation of a wide variety of continuous problems relevant for, e.g., numerical analysis and mathematical modeling.  One classical example is the discretisation of Calder\'on–Zygmund operators using a smooth wavelet orthonormal basis, see \cite{Meyer1997}. A well-localized basis often also  supports a simple characterization of smoothness in terms of certain sparseness conditions relative to expansions in the basis. For example, smoothness measured on the Besov scale can be shown to corresponds to a certain sparseness condition on a corresponding orthonormal wavelet expansion \cite{Meyer1992}. Moreover, such norm characterizations often also allow us to make a natural connection to $m$-term nonlinear approximation using the basis by identifying certain smoothness spaces as nonlinear approximation spaces \cite{Garrigos2004,Gribonval2004,Kyriazis2006,Kyriazis2002}. As a consequence we may gain specific knowledge on how to compress smooth functions using sparse representations of the functions in the unconditional basis \cite{DeVore1992,DeVore1992a}.

From the point of view of mathematical modeling, it is desirable to have as flexible tools as possible as it allows one to incorporate more  refined structure of various real-world phenomena in the model. 
Function spaces in anisotropic and mixed-norm  settings have therefore attached considerable interest recently, see for example \cite{MR3377120,MR3707993,MR2720206,MR2186983} and references therein. 

However, when it comes to construction of  unconditional bases, a serious design challenge is to make the basis compatible with a desired anisotropic structure. One illustrative such example is the challenge of designing orthonormal wavelet bases for an arbitrary dilation matrix, see, e.g., \cite{bownik_construction_2001,MR2250142}. An addition design challenge comes into play when one also desires that the constructed basis should  be able to capture smoothness relative to some predefined notion of smoothness. Often this constraint translates to a specific requirement on the  time-frequency structure of the basis. For example, Besov spaces are constructed using a dyadic decomposition of the frequency space, while the family of modulation spaces introduced by Feichtinger \cite{Feichtinger2003} correspond to a uniform decomposition of the frequency space. So for characterization of Besov spaces one needs a wavelet-type basis with a dyadic frequency structure in order to characterize smoothness, while for modulation spaces one needs a Gabor type system compatible with a uniform decomposition of the frequency space. 

 In the present paper we address some of the mentioned challenges by proposing a general procedure for   
construction of a large and flexible family of   tensor product orthonormal bases for $L_2(\bR^d)$ that can be adapted to be compatible with an arbitrary anisotropy on $\bR^d$.  We will show that the construction can easily be adapted to provide  unconditional bases for a large and diverse family of smoothness spaces constructed by the so-called decomposition method.

The theory of decomposition spaces \cite{Feichtinger1987,Feichtinger1985} provides a very general framework for 
the study of function spaces that can be adapted to almost any anisotropic setting, and the special case of  decomposition spaces in  the frequency domain \cite{Borup2008}, provide a very flexible model for smoothness that allows for a unified study of classical smoothness spaces such as modulation spaces, (aniotropic) Besov spaces, and (aniotropic)  Triebel-Lizorkin spaces.
In this paper we focus on the family of $\alpha$-modulation and $\alpha$-Triebel–Lizorkin spaces. These two families of smoothness spaces depend on a parameter $\alpha\in [0,1]$ that governs the geometric time-frequency structure of the spaces. For $\alpha=0$, we obtain the modulation spaces, while for $\alpha=1$ we obtain anisotropic Besov and Triebel-Lizorkin spaces, respectively.  For $\alpha\in(0,1)$, we obtain spaces corresponding to a certain ``polynomial'' type decomposition of the time-frequency space. The reader may consult \cite{Borup2008,AlJawahri2019,MR2295293,MR2396847,MR4358697,MR4082240} for further discussion on  $\alpha$-modulation type spaces.

The main contribution of the present paper is to offer an easy explicit
construction of a family of  universal tensor product orthonormal bases for $L_2(\bR^d)$, based on so-called univariate brushlet systems. The univariate brushlets are incredibly flexible orthonormal function systems based on local Fourier bases as introduced by Coifman and Meyer  \cite{MR1089710}, and by Malvar  \cite{56057} for applications in signal processing. These systems were further developed by Wickerhauser \cite{MR1209260}.  Laeng pointed out in \cite{MR1081623} that it is possible to map a local Fourier basis by the Fourier transform to a new basis with compact support in the frequency domain. In \cite{Meyer1997a}, Coifman and F. G. Meyer studied similar systems, coining the term bruslets, using the bases introduced by Wickerhauser.

Let us give an overview of the contributions from our study.

\begin{enumerate}
\item The construction of the tensor brushlet orthonormal bases is ``painless'' in the sense that all parameters for the construction are completely specified for any given  anisotropy, $\alpha$-paramter, and dimension $d\geq 2$.  The construction is presented in Section \ref{sec:bases}. 
\item The construction provides the first example of an orthonormal system extending to universal unconditional bases for the full scale of anisotropic $\alpha$-modulation and $\alpha$-Triebel-Lizorkin type spaces on $\bR^d$, regardless of the dimension $d\geq 2$.

    \item The basis supports a full characterization of the norm in anisotropic $\alpha$-modulation and $\alpha$-Triebel-Lizorkin  spaces. The characterizations are proved in Section \ref{s:onb}.
    \item The non-redundancy of the basis implies that inverse Bernstein estimates for nonlinear $m$-term approximation with the brushlet system can be obtained with approximation error measured in  anisotropic $\alpha$-modulation and $\alpha$-Triebel-Lizorkin   spaces. The approximation results are presented in Section \ref{s:approx}.
\end{enumerate}

\subsubsection*{Notation} Throughout the article, positive constants, denoted by $c$, $c'$ or similar, may vary at every occurrence. For constants where  a dependence on parameters $q_1, q_2, \ldots,q_K,$ is essential for the context, it will be stated as  $c:=c(q_1,q_2,\ldots,q_K)$. We will use the notation $A\asymp B$ to indicate that the quantities $A$ and $B$ are equivalent in the sense that there exists constants $c,c'\in (0,\infty)$ such that $cA\leq B\leq c'A$.

\section{The anisotropic setup and  preliminaries}

Let us introduce the anisotropic structure on $\bR^d$ that will be used in the construction of tensor product brushlet bases and for defining Besov and Triebel-Lizorkin type spaces associated with various polynomial type decompositions of the frequency space $\bR^d$. 

Let $\cb,\cx\in\mathbb{R}^d$ and $t>0.$ We denote by $t^{\cb}\cx:=(t^{b_1}x_1,\dots,t^{b_d}x_d)^\top.$ We fix an anisotropy $\vec{a}\in [1,\infty)^d$, and associate the anisotropic quasi-norm $|\cdot|_{\vec{a}}$ as follows: We put $|0|_{\vec{a}}:=0$ and for $\cx\neq0$ we set $|\cx|_{\vec{a}}:=t_0,$ there $t_0$ is the unique positive number such that $|t_0^{-\vec{a}}\cx|=1.$ One observes immediately that
\begin{equation}\label{ad1}
|t^{\vec{a}}\cx|_{\vec{a}}=t|\cx|_{\vec{a}},\;\;\text{for every}\;\;\cx\in\mathbb{R}^d,\;t>0.
\end{equation}
From this we notice that $|\cdot|_{\vec{a}}$ is not a norm unless  $\vec{a}=(1,\dots,1)^\top$, where it is equivalent to the Euclidean norm $|\cdot|$.
Suppose  $\va=(a_1,\ldots,a_d)^\top$. Then one may verify that, uniformly for $\cx\in\bR^d$,
\begin{equation}\label{eq:equiv_norm}
|\cx|_{\vec{a}}\asymp \sum_{j=1}^d |x_j|^{1/a_j}\asymp \max_{j}|x_j|^{1/a_j}.
\end{equation}
The latter quantity will play an important role in the various constructions below, so for notational convenience, we also define 
\begin{equation}\label{eq:def_norminf}
|\cx|_{\vec{a},\infty}:=\max_{j}|x_j|^{1/a_j},\qquad \cx\in\bR^d.
\end{equation}

The anisotropic distance can be directly compared to the Euclidean norm, see e.g.\ \cite{Bagby:1975uc,Bownik2006}, in the sense that there are  constants $c,c'>0$ such that for every $\cx\in\mathbb{R}^d$, 
\begin{equation}\label{ad8}
c(1+|\cx|_{\vec{a}})^{\gamma_m}\le1+|\cx|\le c'(1+|\cx|_{\vec{a}})^{\gamma_M},
\end{equation}
where we denoted $\gamma_m:=\min_{1\le j\le d}a_j,\;\gamma_M:=\max_{1\le j\le d}a_j.$
 Furthermore, we define the homogeneous dimension by 
\begin{equation}\label{nu}
  \nu:=|\vec{a}|:=a_1+\cdots+a_d,
  \end{equation}
where we notice that we always have $\nu\geq d$ since $\vec{a}\in [1,\infty)^d$.
We will need the following anisotropic bracket. We consider $(1,\vec{a})\in\bR^{d+1}$ and define
$$\brac{\cx}:=|(1,\cx)|_{(1,\vec{a})},\qquad \cx\in\bR^d.$$ 
This quantity has been studied in detail in \cite{Borup2008,Stein1978}.
One may verify that the following equivalence holds  
\begin{equation}\label{bracket}
  \brac{\cx}\asymp  1+|\cx|_{\vec{a}},\qquad \cx\in\bR^d.
  \end{equation}
One can  also show that there is a constant $c_1>0$ such that
\begin{equation}\label{brac_esti}
  \brac{\cx+\cy}\leq c_1 \brac{\cx} \brac{\cy},\qquad \cx,\cy\in\bR^d.
  \end{equation}
For $\cx\in\bR^d$ and $r>0$, we denote by $$B_{\vec{a}}(\cx,r):=\{\cy\in\bR^d:\;|\cx-\cy|_{\vec{a}}<r\},$$ the anisotropic ball of radius $r$, centered at $\cx$. Notice that $B_{\vec{a}}(\cx,r)$ is convex and $|B_{\vec{a}}(\cx,r)|=|B_{\vec{a}}(0,1)|r^\nu$. Similarly, we define the rectangle centered at $\cx$,
$$R_{\vec{a}}(\cx,r):=\{\cy\in\bR^d:\;|\cx-\cy|_{\vec{a},\infty}<r\}.$$
We notice from \eqref{eq:equiv_norm} that there exist $0<c_2\leq c_3$ such that for any $\cx\in\bR^d$, $r>0$,
$$B_{\vec{a}}(\cx,c_2r)\subseteq R_{\vec{a}}(\cx,r)\subseteq B_{\vec{a}}(\cx,c_3r).$$
Finally, we  mention that by changing to polar coordinates, one can easily show that for $\tau>\nu$,
\begin{equation}\label{brac_int}
	\int_{\bR^d} \brac{\cx}^{-\tau}\,\dx\le c_{\tau}<\infty.
\end{equation}
\section{Anisotropic tensor product brushlet orthonormal bases}\label{sec:bases}
In this section we present the main construction of tensor brushlet bases adapted to the given anisotropy $\va$. With a view towards the known  time-frequency structure of $\alpha$-modulation spaces, we will impose  a ``polynomial'' type decomposition structure in the frequency domain. 
We will call on a number of known properties of univariate brushlet systems for the construction, and for the benefit of the reader, we have included a review of these properties in Appendix~\ref{app:brush}.

In the special case $d=2$, two constructions of orthonormal bases compatible with $\alpha$-modu{-}lation spaces have been considered previously. In the purely isotropic case, there is a construction by the  author \cite{Nielsen2010}, and Rasmussen considered the anisotropic case  
\cite{Rasmussen2012}. The construction by Rasmussen is based on a rather involved iterative argument and therefore cannot easily be extended to higher dimensions. The construction presented below places no restriction on $d\geq 2$.

\subsection{Decomposition of the frequency space $\bR^d$} We first consider the following subsets of the real axis, with endpoints that are compatible with standard univariate $\alpha$-coverings, see, e.g., \cite{Borup2006b}.
Fix a parameter $\beta\geq 1$ to be determined later. We first focus on creating a suitably calibrated subdivision of the the ``corridors'' defined for $j\geq 1$ by
$$K_j:=\{\cx\in\bR^d: j^{\beta}\leq |\cx|_{\vec{a},\infty}<(j+1)^{\beta}\}.$$
Suppose that $\va=(a_1,\ldots,a_d)^\top$ and recall that $\nu:=a_1+\cdots+a_d$. For $j\geq 1$, and $i\in\{1,\ldots, d\}$, we let
\begin{equation}\label{eq:Aj}
A_j^i:=[-(j+1)^{a_i\beta},(j+1)^{a_i\beta}),\qquad B_j^i:= \big[-j^{a_i\beta} ,j^{a_i\beta} \big).
\end{equation}
Let us first consider 
\begin{equation}\label{eq:Bj}
A_j^i\backslash B_j^i=\big[-(j+1)^{a_i\beta} ,-j^{a_i\beta} \big)\cup 
\big[j^{a_i\beta} ,(j+1)^{a_i\beta} \big).
\end{equation}
 For $t\in \bR$, we let $\lceil t\rceil$ denote the least integer greater than or equal to $t$.
 We  subdivide $(A_j^i\backslash B_j^i)\cap [0,\infty)$ into $\lceil j^{ a_i-1}\rceil$
 half-open intervals $I_{j,k}^{i,O}=[\alpha_{j,k}^{i,L},\alpha_{j,k}^{i,R})$ of equal length, 
 and similarly, 
 we subdivide $(A_j^i\backslash B_j^i)\cap (-\infty,0]$ 
by the intervals $I_{j,-k}^{i,O}:=\big[-\alpha_{j,k}^{i,R},-\alpha_{j,k}^{i,L}\big)$, $k=1,\ldots,\lceil j^{ a_i-1}\rceil$. 
We obtain the partition
$$
A_j^i\backslash B_j^i=I_{j,-\lceil j^{ a_i-1}\rceil}^{i,O}\cup \cdots\cup I_{j,-1}^{i,O}
\cup 
I_{j,1}^{i,O}\cup \cdots\cup I_{j,\lceil j^{ a_i-1}\rceil}^{i,O}.$$
We then subdivide $B_j^i$ into $\lceil j^{ a_i}\rceil$
 half-open intervals $I_{j,k}^{i,I}$, $k=1,\ldots, \lceil j^{ a_i}\rceil$,
 where we notice 
that $$|I_{j,k}^{i,I}|\asymp j^{\beta a_i}/j^{a_i}=j^{a_i(\beta-1)},$$ uniformly in $j$ and $k\in \{1,\ldots,\lceil j^{ a_i}\rceil\}$. This provides a partition of $A_j^i$,
\begin{equation}\label{eq:partt}
A_j^i = 
I_{j,-\lceil j^{ a_i-1}\rceil}^{i,O}\cup \cdots\cup I_{j,-1}^{i,O}\cup
I_{j,1}^{i,I}\cup
\cdots \cup I_{j,\lceil j^{ a_i}\rceil}^{i,I}\cup I_{j,1}^{i,O}\cup \cdots\cup I_{j,\lceil j^{ a_i-1}\rceil}^{i,O},
\end{equation}
where each of the $O(j^{ a_i})$ half-open intervals in the partition has length $\asymp j^{a_i(\beta-1)}$.
Let us define 
$$\mathcal{A}_j^i=\{I_{j,\pm k}^{i,O}: k=1,\ldots,\lceil j^{ a_i}\rceil\}
\cup \{I_{j,k}^{i,I}: k=1,\ldots, \lceil j^{ a_i}\rceil\},$$
and the corresponding product set
$$\mathcal{A}_j:=\mathcal{A}_j^1\times \mathcal{A}_j^2\times\cdots\times \mathcal{A}_j^d.$$
Similarly, we define $\mathcal{B}_j^i\subset \mathcal{A}_j^i$ by
$$\mathcal{B}_j^i= \{I_{j,k}^{i,I}: k=1,\ldots, \lceil j^{ a_i}\rceil\},$$
and the corresponding product set $\mathcal{B}_j\subset \mathcal{A}_j$,
$$\mathcal{B}_j:=\mathcal{B}_j^1\times \mathcal{B}_j^2\times\cdots\times \mathcal{B}_j^d.$$
We put $\mathcal{K}_j:=\mathcal{A}_j\backslash \mathcal{B}_j$ and 
 fix some ordering of the $|\mathcal{K}_j|$ rectangles in $\mathcal{K}_j$, say, \begin{equation}\label{eq:Kj}
    \mathcal{K}_j=\big\{R^j_1,R_2^j,\ldots, R^j_{|\mathcal{K}_j|}\big\}.
\end{equation}
For  notational convenience, we let $\mathcal{K}_0:=\{R^0_1\}$, with $R_1^0:=[-1,1)^d$. We also notice that the chosen orderings from \eqref{eq:Kj} naturally induce an ordering on the full collection
\begin{equation}\label{eq:K}
\mathcal{K}:=\mathcal{K}_0\cup \mathcal{K}_1\cup \mathcal{K}_2\cup\cdots.
\end{equation}

Next, we notice from Eqs.\ \eqref{eq:Aj}, \eqref{eq:Bj}, and \eqref{eq:partt} that the rectangles in $\mathcal{K}_j$ in fact form a partition of $K_j$,
$$K_j=\bigcup_{R^j_k\in \mathcal{K}_j} R^j_k.$$
 
For $j\geq 1$ and $k\in \{1,\ldots, |\mathcal{K}_j|\}$, we let $\cc_k^j\in\bR^d$ denote the center of the rectangle $R^j_k\in\mathcal{K}_j$, i.e., if $R^j_k=I_1\times\cdots \times I_d$ with $I_j=[a_i,b_i)$ then $\cc_k^j:=\big(\frac{b_1-a_1}{2},\ldots,\frac{b_d-a_d}{2}\big)$. We define an associated affine transform by 
\begin{equation}\label{eq:affine}
 T_k^i(\cdot):=j^{(\beta-1)\vec{a}} \cdot+\cc_k^j,
   \end{equation}
Again for notational convenience, we  let $T_1^0:=\text{Id}_{\bR^d}$, and consequently, $\cc_1^0=0$.
An illustration of the construction is given in Figure \ref{fig:1}, while some of the needed geometric properties of the construction  are derived in Lemma \ref{le:cut} below.

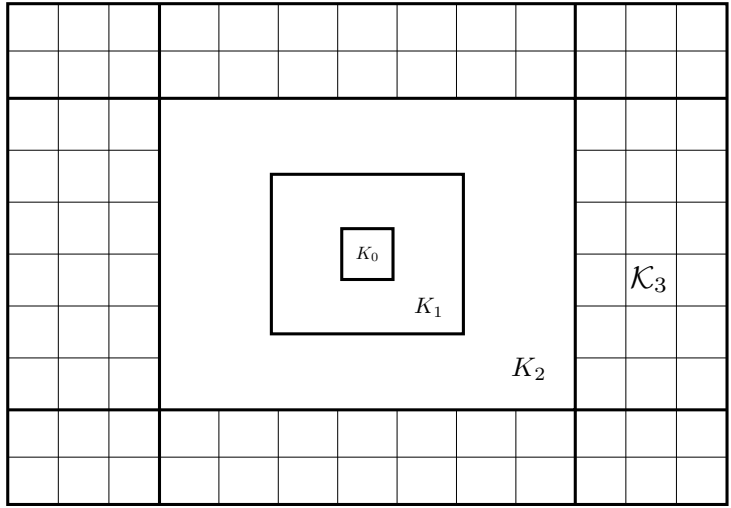
\begin{figure}[ht]
\resizebox{9.5cm}{!}{%
\begin{tikzpicture}
  \draw[line width=1.2mm] (-1,1) -- (1,1) -- (1,-1) -- (-1,-1) -- cycle ;
  \draw[line width=1.2mm] 
  (3.745753310, 3.138336392) -- (-3.745753310, 3.138336392) -- (-3.745753310, -3.138336392) -- (3.745753310, -3.138336392) --cycle;
  \draw[line width=1.2mm] 
  (8.110321686, 6.127030896)--(-8.110321686, 6.127030896)--(-8.110321686, -6.127030896)--(8.110321686, -6.127030896)-- cycle; 
\draw[line width=1.2mm] (14.03066786, 9.849155307) -- (-14.03066786, 9.849155307) --(-14.03066786, -9.849155307) --(14.03066786, -9.849155307) --cycle; 

\draw (8.110321686, -6.127030896+2.042343632) -- (14.03066786, -6.127030896+2.042343632);
\draw (8.110321686, -6.127030896+2*2.042343632) -- (14.03066786, -6.127030896+2*2.042343632);
\draw (8.110321686, -6.127030896+3*2.042343632) -- (14.03066786, -6.127030896+3*2.042343632);
\draw (8.110321686, -6.127030896+4*2.042343632) -- (14.03066786, -6.127030896+4*2.042343632);
\draw (8.110321686, -6.127030896+5*2.042343632) -- (14.03066786, -6.127030896+5*2.042343632);
\draw (-8.110321686, -6.127030896+2.042343632) -- (-14.03066786, -6.127030896+2.042343632);
\draw (-8.110321686, -6.127030896+2*2.042343632) -- (-14.03066786, -6.127030896+2*2.042343632);
\draw (-8.110321686, -6.127030896+3*2.042343632) -- (-14.03066786, -6.127030896+3*2.042343632);
\draw (-8.110321686, -6.127030896+4*2.042343632) -- (-14.03066786, -6.127030896+4*2.042343632);
\draw (-8.110321686, -6.127030896+5*2.042343632) -- (-14.03066786, -6.127030896+5*2.042343632);

\draw (-8.110321686+2.317234767, 6.127030896) -- (-8.110321686+2.317234767, 9.849155307);
\draw (-8.110321686+2.317234767, -6.127030896) -- (-8.110321686+2.317234767, -9.849155307);

\draw (-8.110321686+2*2.317234767, 6.127030896) -- (-8.110321686+2*2.317234767, 9.849155307);
\draw (-8.110321686+2*2.317234767, -6.127030896) -- (-8.110321686+2*2.317234767, -9.849155307);

\draw (-8.110321686+3*2.317234767, 6.127030896) -- (-8.110321686+3*2.317234767, 9.849155307);
\draw (-8.110321686+3*2.317234767, -6.127030896) -- (-8.110321686+3*2.317234767, -9.849155307);

\draw (-8.110321686+4*2.317234767, 6.127030896) -- (-8.110321686+4*2.317234767, 9.849155307);
\draw (-8.110321686+4*2.317234767, -6.127030896) -- (-8.110321686+4*2.317234767, -9.849155307);
\draw (-8.110321686+5*2.317234767, 6.127030896) -- (-8.110321686+5*2.317234767, 9.849155307);
\draw (-8.110321686+5*2.317234767, -6.127030896) -- (-8.110321686+5*2.317234767, -9.849155307);

\draw (-8.110321686+6*2.317234767, 6.127030896) -- (-8.110321686+6*2.317234767, 9.849155307);
\draw (-8.110321686+6*2.317234767, -6.127030896) -- (-8.110321686+6*2.317234767, -9.849155307);
  
  \draw[line width=1.2mm] (8.110321686, 6.127030896) -- (14.03066786,6.127030896);
    \draw[line width=1.2mm] (8.110321686, -6.127030896) -- (14.03066786,-6.127030896);
      \draw[line width=1.2mm] (-8.110321686, 6.127030896) -- (-14.03066786,6.127030896);
        \draw[line width=1.2mm] (-8.110321686, -6.127030896) -- (-14.03066786,-6.127030896);
         \draw[line width=1.2mm] (8.110321686, 6.127030896) -- (8.110321686, 9.849155307);
          \draw[line width=1.2mm] (-8.110321686, 6.127030896) -- (-8.110321686, 9.849155307);
           \draw[line width=1.2mm] (-8.110321686, -6.127030896) -- (-8.110321686, -9.849155307);
            \draw[line width=1.2mm] (8.110321686, -6.127030896) -- (8.110321686, -9.849155307);
        \draw (8.110321686+1.973448725,-9.849155307) -- (8.110321686+1.973448725,9.849155307);    
          \draw (8.110321686+2*1.973448725,-9.849155307) -- (8.110321686+2*1.973448725,9.849155307); 
           \draw (-8.110321686-1.973448725,-9.849155307) -- (-8.110321686-1.973448725,9.849155307);    
          \draw (-8.110321686-2*1.973448725,-9.849155307) -- 
          (-8.110321686-2*1.973448725,9.849155307);

\draw  (-14.03066786,6.127030896+1.861062206) --(14.03066786,6.127030896+1.861062206);
\draw  (-14.03066786,-6.127030896-1.861062206) --(14.03066786,-6.127030896-1.861062206);
              \node[scale=2] at (0,0) {${K}_0$};
                                       \node[scale=2.5] at (2.4,-2.1) {${K}_1$};
                          \node[scale=3] at (6.3,-4.5) {${K}_2$};
                              \node[scale=3.5] at (11,-1) {$\mathcal{K}_3$};
                           
\end{tikzpicture}}
\caption{The corridors ${K}_0, {K}_1$, ${K}_2$, and ${K}_3$ in $\bR^2$ for $\beta=1.1$ and $\vec{a}=(\sqrt{3},\frac{3}{2})$. The fine grid represents the sets in the partition $\mathcal{K}_3$ of $K_3$. 
Notice that $\lceil 3^{\sqrt{3}-1}\rceil=3$, while $\lceil 3^{3/2-1}\rceil=2$, leading to the illustrated $3\times 2$ sub-grid for the corner set $[3^{1.1\cdot \sqrt{3}}, 4^{1.1\cdot \sqrt{3}})\times [3^{1.1\cdot 1.5},4^{1.1\cdot 1.5})$, which is compatible with the anisotropy $\vec{a}$. 
}\label{fig:1}
\end{figure}

We have the following result summarizing some of important properties of $\mathcal{K}$.
\begin{mylemma}\label{le:cut}
The following holds true about the partitions $\mathcal{A}_j^i$ defined in Eq.\ \eqref{eq:partt}.
\begin{enumerate}
    \item[(a)] 
For $i\in\{1,2,\ldots,d\}$,  there exist  positive constants $c_{i,1}:=c_{i,1}(a_i,\beta)$ and $c_{i,2}:=c_{i,2}(a_i,\beta)$ such that for $j\geq 1$,
$$  c_{i,1} j^{a_i(\beta-1)}\leq  |I| \leq c_{i,2} j^{a_i(\beta-1)},\qquad I\in \mathcal{A}_j^i.$$
\item[(b)]  There exist positive constants $c_3:=c_3(a_i,\beta)$ and $c_4:=c_4(a_i,\beta)$ such that for $j\geq 1$,
 $$c_3 j^{\nu(\beta-1)}\leq  |R| \leq c_4 j^{\nu(\beta-1)},\qquad R\in \mathcal{A}_j.$$
 
 \item[(c)] There exist positive constants $c_5:=c_5(a_i,\beta)\leq c_6:=c_6(a_i,\beta)$
 such that for the affine transformations $T_k^j$ defined in Eq.\ \eqref{eq:affine}, we have
  \begin{equation*}
  T_k^j(c_5 [-1,1]^d)\subseteq R_k^j \subseteq T_k^j(c_6 [-1,1]^d),\qquad j\geq 1, k\in \{1,\ldots,|\mathcal{K}_j|\}.
  \end{equation*}

\item[(d)] Let $\alpha\in [0,1)$ and put $\beta=\frac{1}{1-\alpha}\geq 1$. Then the following geometric relation holds, uniformly in $j\geq 1$,
 \begin{equation*}
     \cx\in R\in \mathcal{K}_j\Longrightarrow |R|\asymp |\cx|_{\vec{a}}^{\nu \alpha}\asymp\brac{\cx}^{\nu\alpha}. 
 \end{equation*}
\end{enumerate}
\end{mylemma}
\begin{proof}
For (a), we first notice that for any $j\geq 1$, using the mean-value theorem,
    $$ a_i\beta \cdot j^{a_i\beta-1} \leq (j+1)^{a_i\beta}-j^{a_i\beta}\leq a_i\beta(j+1)^{a_i\beta-1},$$
  so, using the notation from \eqref{eq:partt}, for  $k\in \{1,\ldots, \lceil j^{ a_i-1}\rceil\}$,
$$
 |I_{j,\pm k}^{i,O}|\leq \frac{a_i\beta(j+1)^{a_i\beta-1}}{\lceil j^{ a_i-1}\rceil}\leq  \frac{a_i\beta(j+1)^{a_i\beta-1}}{ j^{ a_i-1}}.$$
We notice that
\begin{equation}\label{eq:ratio}
    \frac{(j+1)^{a_i\beta-1}}{j^{a_i\beta-1}}=\bigg(1+\frac{1}{j}\bigg)^{a_i\beta-1}\leq 2^{a_i\beta-1},
\end{equation}  
which implies that 
  \begin{equation}\label{eq:up1}
  |I_{j,\pm k}^{i,O}|\leq \frac{a_i\beta(j+1)^{a_i\beta-1}}{ j^{ a_i-1}}\leq 
 2^{a_i\beta-1} {a_i\beta\cdot j^{a_i(\beta-1)}}.
  \end{equation}
Also, since $j\geq 1$ and $a_i\geq 1$,
\begin{equation}
    \label{eq:up2}
  |I_{j,\pm k}^{i,O}|\geq \frac{a_i\beta \cdot j^{a_i\beta-1}}{ \lceil j^{ a_i-1}\rceil}
\geq \frac{a_i\beta \cdot j^{a_i\beta-1}}{ j^{ a_i-1}+1}\geq \frac{1}{2}\frac{a_i\beta \cdot j^{a_i\beta-1}}{ j^{ a_i-1}}= \frac{a_i\beta}{2}j^{a_i(\beta-1)}.
\end{equation}
For the remaining intervals from $\mathcal{A}_j^i$, we have for $j\geq 1$, $k=1,\ldots, \lceil j^{ a_i}\rceil$,
 \begin{equation}
    \label{eq:up3}
  |I_{j,k}^{i,I}|=\frac{2j^{a_i\beta}}{\lceil j^{ a_i}\rceil}
\leq \frac{2j^{a_i\beta}}{j^{ a_i}}=2j^{a_i(\beta-1)},
\end{equation}
 and, similarly, using that $j\geq 1$ and $a_i\geq 1$,
  \begin{equation}
    \label{eq:up4}
 |I_{j,k}^{i,I}|=\frac{2j^{a_i\beta}}{\lceil j^{ a_i}\rceil}\geq 
 \frac{2j^{a_i\beta}}{j^{ a_i}+1}\geq \frac{1}{2}\frac{2j^{a_i\beta}}{j^{ a_i}}=
\frac{j^{a_i(\beta-1)}}{2}. 
\end{equation}
Now, by \eqref{eq:up1} and \eqref{eq:up3}, we may put $c_{i,2}=\max\{2,2^{a_i\beta-1}\}$, and, by  \eqref{eq:up2} and \eqref{eq:up4}, we let $c_{i,1}=\min\{\frac{1}{2},\frac{a_i\beta}{2}\}$. This proves (a).

We turn to the proof of (b), which is a direct consequence of  (a). We consider a rectangle $R\in \mathcal{A}_j$. Recall that $R=I_1\times\cdot\times I_d$, $I_i\in  \mathcal{A}_j^i$, so by using the estimate from (a) for each $I_i$, 
$$\bigg(\prod_{j=1}^d c_{i,1}\bigg)j^{\nu(\beta-1)}=\prod_{j=1}^d c_{i,1} j^{a_i(\beta-1)}\leq |R|\leq  \prod_{j=1}^d c_{i,2} j^{a_i(\beta-1)}=\bigg(\prod_{j=1}^d c_{i,2}\bigg)j^{\nu(\beta-1)},$$
and we simply put $c_3:=\prod_{j=1}^d c_{i,1}$ and $c_4:=\prod_{j=1}^d c_{i,2}$.

Next, we consider (c), which is also a consequence of (a). We first notice that for $C>0$,
$$T_k^j(C[-1,1]^d)=[-Cj^{a_1(\beta-1)},Cj^{a_2(\beta-1)}]\times \cdots\times [-Cj^{a_d(\beta-1)},Cj^{a_d(\beta-1)}]+c_k^j.$$
The interval diameter in coordinate $i$ is thus $2Cj^{a_i(\beta-1)}$ and comparing this to the estimates from (a), we see that we may take $c_5:=\frac{1}{2}\min\{c_{1,1},c_{2,1},\ldots,c_{d,1}\}$ and $c_6:=\frac{1}{2}\max\{c_{1,2},c_{2,2},\ldots,c_{d,2}\}$, to obtain
\begin{equation*}
  T_k^j(c_5 [-1,1]^d)\subseteq R_k^j \subseteq T_k^j(c_6 [-1,1]^d),\qquad j\geq 1, k\in \{1,\ldots,|\mathcal{K}_j|\}.
  \end{equation*}

Finally, for (d), we let $j\geq 1$ and take any rectangle $R\in \mathcal{K}_j$, and let $\cx\in R$. Notice that $R\subset K_j$, so by the construction of $K_j$, 
$$j^{\beta}\leq |\cx|_{\vec{a},\infty}<(j+1)^{\beta}.$$
Hence, for $\beta=\frac{1}{1-\alpha}$, using (b),
$$\frac{1}{c_4}\frac{j^{\nu\frac{\alpha}{1-\alpha}}}{ j^{\nu(\beta-1)}}\leq \frac{|\cx|_{{\vec{a},\infty}}^{\nu \alpha}}{|R|}\leq \frac{1}{c_3}\frac{(j+1)^{\nu\frac{\alpha}{1-\alpha}}}{ j^{\nu(\beta-1)}}.$$
We notice that $\beta-1=\frac{\alpha}{1-\alpha}$, and by an estimate similar to the estimate from Eq.\ \eqref{eq:ratio}, we obtain  $(j+1)^{\nu\frac{\alpha}{1-\alpha}}\leq 2^{\nu\frac{\alpha}{1-\alpha}}j^{\nu\frac{\alpha}{1-\alpha}}$,  so we may conclude that 
$$0<\frac{1}{c_4}\leq \frac{|\cx|_{{\vec{a},\infty}}^{\nu \alpha}}{|R|}\leq \frac{2^{\nu\frac{\alpha}{1-\alpha}}}{c_3}<+\infty.$$
Recall that $|\cdot|_{\vec{a}}\asymp |\cdot|_{\vec{a},\infty}$, and by \eqref{bracket} it follows that $ |\cx|_{\vec{a}}\asymp \brac{\cx}$ uniformly for $|\cx|_{\vec{a},\infty}\geq 1$. This proves (d) and also concludes the proof of the lemma.
\end{proof}

\subsection{Tensor-product brushlet bases} Let us now define a corresponding tensor-product brushlet system. We shall use the notation $\bN_0:= \{0\}\cup \bN$.
By Lemma \ref{le:cut}, there exists $c_i:=c_i(a_i,\beta)\in (0,1/2)$, for $i=1,2,\ldots,d$, such that
 by selecting the common cutoff radius $c_i j^{a_1(\beta-1)}$ at all inner knots 
 in the  partition specified in Eq.\ \eqref{eq:partt}, and at 
 the endpoints $\pm (j+1)^{a_i\beta}$, we  select cutoff radii $c_i (j+1)^{a_1(\beta-1)}$,  
 the cutoff radii conditions given in  Eq.\ \eqref{eq:epsilon} are satisfied for any $j\geq 1$.

With the $\epsilon$-values so assigned to the intervals from Eq.\ \eqref{eq:partt}, for each $j\in\bN$ and $R_{j,k}=I^1_{j,k}\times I^2_{j,k}\times\cdots\times I_{j,k}^d\in \mathcal{K}_j$,  we associate the orthonormal brushlet system
$$\mathcal{W}_{j,k}:=\{w_{R_{j,k},\cn}:\cn\in \bN_0^d\},$$
with $w_{R_{j,k},\cn}$ defined as the following tensor product  $$w_{R_{j,k},\cn}:=\bigotimes_{i=1}^d w_{n_i,I^i_{j,k}},$$
for $\cn=(n_1,\ldots,n_d)^T$. 
We define a low-pass part by $\mathcal{K}_0=\{[-1,1)^d\}$, where we for coordinate $i$ assign left and right  cutoff radii $c_i$ to the interval  $[-1,1)$. The corresponding low-pass brushlet system is denoted
$$\mathcal{W}_{0,0}:=\{w_{R_{0,0},\cn}:\cn\in \bN_0^d\}.$$

We have the following fundamental properties of the system $\{w_{R,\cn}\}_{R,\cn}$.

\begin{myprop}\label{prop:onb}
For $\beta\geq 1$, the following holds true.
\begin{enumerate}
    \item[(i).] For $j\geq 1$, the family $$\mathcal{W}_j:=\bigcup_{k=1}^{|\mathcal{K}_j|}\mathcal{W}_{j,k}$$ is orthonormal in $L_2(\bR^d)$.
\item[(ii).] For $j\geq 1$, let $P_{\mathcal{W}_j}$ denote the orthogonal projection onto the $L_2(\bR^d)$-closure of $\text{Span}(\mathcal{W}_j)$. Then
$$P_{\mathcal{W}_j}=P_{A_{j}^1}\otimes P_{A_{j}^2}\otimes\cdots\otimes P_{A_{j}^d}-P_{B_{j}^1}\otimes P_{B_{j}^2}\otimes\cdots\otimes P_{B_{j}^d},$$
 where  $A_j^i:=\big[-(j+1)^{\beta a_i},(j+1)^{\beta a_i}\big)$ with left and right  cutoff radii
 $c_i(j+1)^{a_i (\beta-1)}$, and $B_j^i:=\big[-j^{\beta a_i},j^{\beta a_i}\big)$ with left and right  cutoff radii
 $c_ij^{a_i (\beta-1)}$. 
\item[(iii).] The full brushlet system $$\mathcal{W}_\beta:=\mathcal{W}_{0,0}\cup \{\mathcal{W}_{j,k}:k=1,2,\ldots,|\mathcal{K}_j|, j\geq 1\}$$
forms an orthonormal basis for $L_2(\bR^d)$.
\end{enumerate}

\end{myprop}
\begin{proof}
For (i), we simply notice that by construction $\mathcal{K}_j\subset \mathcal{A}_j$, making $\mathcal{W}_j$ a subset of the larger system 
$$\mathcal{S}_j:=\{w_{R,\cn}:\cn\in \bN_0^d, R\in\mathcal{A}_j\}.$$
By the product structure $\mathcal{A}_j:=\mathcal{A}_j^1\times\cdots\times \mathcal{A}_j^d$, it follows easily that $\mathcal{S}_j$ is a $d$-fold tensor product of univariate orthonormal brushlet systems in $L_2(\bR)$, making  $\mathcal{S}_j$ an orthonormal system in $L_2(\bR^d)$, and $\mathcal{W}_j\subset \mathcal{S}_j$ is thus an orthonormal system in $L_2(\bR^d)$.

 For (ii), we first use the product structure $\mathcal{A}_j:=\mathcal{A}_j^1\times\cdots\times \mathcal{A}_j^d$ to obtain,
\begin{align}\sum_{R\in\mathcal{A}_j} P_{R}
&=
\sum_{I_1\in \mathcal{A}_j^1,I_2\in \mathcal{A}_j^2,\ldots, I_d\in \mathcal{A}_j^d} {\mathcal{P}}_{I_1}\otimes {\mathcal{P}}_{I_2}\otimes\cdots\otimes {\mathcal{P}}_{I_{d}}\nonumber\\
&=
\sum_{I_1\in \mathcal{A}_j^1}\sum_{I_2\in\mathcal{A}_j^2}\cdots
\sum_{I_d\in \mathcal{A}_j^d}
 {\mathcal{P}}_{I_1}\otimes {\mathcal{P}}_{I_2}\otimes\cdots\otimes {\mathcal{P}}_{I_{d}}\nonumber\\
&=\bigg(\sum_{I_1\in \mathcal{A}_j^1} {\mathcal{P}}_{I_1} \bigg)\otimes \bigg(\sum_{I_2\in \mathcal{A}_j^2} {\mathcal{P}}_{I_2} \bigg)\otimes\cdots\otimes \bigg(\sum_{I_d\in \mathcal{A}_j^d} {\mathcal{P}}_{I_d} \bigg)\nonumber\\
&={\mathcal{P}}_{A_{j}^1}\otimes {\mathcal{P}}_{A_{j}^2}\otimes\cdots\otimes {\mathcal{P}}_{A_{j}^d},\label{eq:teleA}
\end{align}
 where we have used the projection addition property stated in Eq.\  \eqref{eq:proj} repeatedly for the final equality. Also, according to Eq.\ \eqref{eq:proj}, the left and right right cutoff radii of $A_{j}^i$ associated with the projection are both 
 $c_i(j+1)^{a_i (\beta-1)}$. The same argument applied to $\mathcal{B}_j:=\mathcal{B}_j^1\times \mathcal{B}_j^2\times\cdots\times \mathcal{B}_j^d$ yields,
 \begin{align}\sum_{R\in\mathcal{A}_j} P_{R}
&={\mathcal{P}}_{B_{j}^1}\otimes {\mathcal{P}}_{B_{j}^2}\otimes\cdots\otimes {\mathcal{P}}_{B_{j}^d},\label{eq:teleB}
\end{align}
with $c_ij^{a_i (\beta-1)}$ as left and right cutoff radii of $B_{j}^i$ for the projection. We combine \eqref{eq:teleA} and \eqref{eq:teleB} to conclude that
\begin{align}
P_{\mathcal{W}_j}&=\sum_{R\in \mathcal{K}_j=\mathcal{A}_j\backslash \mathcal{B}_j} P_R\nonumber\\&=\sum_{R\in \mathcal{A}_j} P_R
-\sum_{R\in \mathcal{B}_j} P_R\nonumber\\
&={\mathcal{P}}_{A_{j}^1}\otimes {\mathcal{P}}_{A_{j}^2}\otimes\cdots\otimes {\mathcal{P}}_{A_{j}^d}
-{\mathcal{P}}_{B_{j}^1}\otimes {\mathcal{P}}_{B_{j}^2}\otimes\cdots\otimes {\mathcal{P}}_{B_{j}^d}.\label{eq:tele}
\end{align}

Finally, to prove (iii), we first notice  for $j\geq 1$,
\begin{align*}
P_{\mathcal{W}_j}P_{\mathcal{W}_{j+1}}&
=\bigg(\bigotimes_{j=1}^d\mathcal{P}_{A_{j}^i}-\bigotimes_{j=1}^d\mathcal{P}_{B_{j}^i}\bigg)
\bigg(\bigotimes_{j=1}^d\mathcal{P}_{A_{j+1}^i}-\bigotimes_{j=1}^d\mathcal{P}_{B_{j+1}^i}\bigg)\\
&=\bigotimes_{j=1}^d\mathcal{P}_{A_{j}^i}\mathcal{P}_{A_{j+1}^i}
- \bigotimes_{j=1}^d\mathcal{P}_{A_{j}^i}\mathcal{P}_{B_{j+1}^i}
- \bigotimes_{j=1}^d\mathcal{P}_{B_{j}^i}\mathcal{P}_{A_{j+1}^i}
+\bigotimes_{j=1}^d\mathcal{P}_{B_{j}^i}\mathcal{P}_{B_{j+1}^i}
\end{align*}
Now, we recall that by construction, ${B_{j+1}^i}={A_{j}^i}$.
Also, by the choice $c_ij^{a_i(\beta-1)}$ of cutoff radii for  $\mathcal{P}_{A_{j}^i}$ and $\mathcal{P}_{A_{j+1}^i}$,
we may call on Eqs.\ \eqref{eq:nested1} and \eqref{eq:nested2} to obtain 
$$\mathcal{P}_{A_{j}^i}\mathcal{P}_{A_{j+1}^i}
=\mathcal{P}_{A_{j+1}^i}\mathcal{P}_{A_{j}^i}=\mathcal{P}_{A_{j}^i}.$$
Since $\mathcal{P}_{A_{j}^i}$ is an orthogonal projection, $\mathcal{P}_{A_{j}^i}^2=\mathcal{P}_{A_{j}^i}$ and we obtain
\begin{align*}
P_{\mathcal{W}_j}P_{\mathcal{W}_{j+1}}
&=\bigotimes_{j=1}^d\mathcal{P}_{A_{j}^i}\mathcal{P}_{A_{j+1}^i}
- \bigotimes_{j=1}^d\mathcal{P}_{A_{j}^i}\mathcal{P}_{B_{j+1}^i}
- \bigotimes_{j=1}^d\mathcal{P}_{B_{j}^i}\mathcal{P}_{A_{j+1}^i}
+\bigotimes_{j=1}^d\mathcal{P}_{B_{j}^i}\mathcal{P}_{B_{j+1}^i}\\
&=\bigotimes_{j=1}^d\mathcal{P}_{A_{j}^i}
- \bigotimes_{j=1}^d\mathcal{P}_{A_{j}^i}
- \bigotimes_{j=1}^d\mathcal{P}_{B_{j}^i}
+\bigotimes_{j=1}^d\mathcal{P}_{B_{j}^i}\\
&=0.
\end{align*}
The same type of argument shows that $P_{\mathcal{W}_{j+1}}P_{\mathcal{W}_j}=0$, and also $P_{\mathcal{W}_{0,0}}P_{\mathcal{W}_1}=P_{\mathcal{W}_1}P_{\mathcal{W}_{0,0}}=0$, so we have orthogonality between all ``levels'', which also shows that the full brushlet system $\mathcal{W}_\beta$ is orthonormal. All that remains is to prove completeness. We will rely on the telescoping structure derived in Eq.\ \eqref{eq:tele}. We have
\begin{align}
P_{\mathcal{W}_{0,0}}+\sum_{j=1}^N P_{\mathcal{W}_j}&=
P_{\mathcal{W}_{0,0}}+\sum_{j=1}^N \big\{{\mathcal{P}}_{A_{j}^1}\otimes {\mathcal{P}}_{A_{j}^2}\otimes\cdots\otimes {\mathcal{P}}_{A_{j}^d}
-{\mathcal{P}}_{B_{j}^1}\otimes {\mathcal{P}}_{B_{j}^2}\otimes\cdots\otimes {\mathcal{P}}_{B_{j}^d}\big\}\nonumber\\
&={\mathcal{P}}_{A_N^1}\otimes {\mathcal{P}}_{A_{N}^2}\otimes\cdots\otimes {\mathcal{P}}_{A_{N}^d},
 \end{align}
 where we used that $B_{1}^i=[-1,1)$ with left and right $\epsilon$-values $c_i$,  $i=1,\ldots, d$, so $$P_{\mathcal{W}_{0,0}}-{\mathcal{P}}_{B_{1}^1}\otimes {\mathcal{P}}_{B_{1}^2}\otimes\cdots\otimes {\mathcal{P}}_{B_{1}^d}=0.$$
It thus follows easily, using the observation in Eq.\ \eqref{eq:summm}, that 
$$P_{R_{0,0}}+\sum_{j=1}^{N}\sum_{R\in \mathcal{K}_j} P_{R_{j,k}}\longrightarrow Id_{L_2(\bR^d)},$$
in the strong operator topology, as $N\rightarrow \infty$. This completes the proof.
\end{proof}

\section{Modulation and Triebel-Lizorkin type smoothness spaces}\label{s:onb}
In the previous section we have constructed a large flexible family of orthonormal bases, all compatible with various ``polynomial'' decompositions of the frequency space $\bR^d$. In this section, we utilize the polynomial structure to decompose anisotropic smoothness spaces with a compatible time-frequency  structure. 

Let us first introduce the notion of anisotropic $\alpha$-coverings for $\alpha\in [0,1]$. For a bounded subset $Q\subset \bR^d$ with non-empty interior, we let 
\begin{align*}
  r_Q &= \sup\{ r\in\bR_+\colon B_{\vec{a}}(\cy,r)\subset Q\; \text{for some}\; \cy\in \bR^d\},\\
  R_Q &= \inf\{ R\in\bR_+\colon Q\subset B_{\vec{a}}(\cy,R)\; \text{for some}\; \cy\in \bR^d\}
\end{align*}
denote, respectively, the radius of the inscribed and circumscribed
anisotropic ball of $Q$. 

We have the following definition.

\begin{mydef}\label{def:1} 
\begin{itemize}
\item[]
    \item[i.]
A countable set $\qp$ of bounded subsets $Q\subset \bR^d$ with non-empty interior is called an
admissible covering if $\bR^d=\cup_{Q\in \qp} Q$ and there
exists $n_0<\infty$ such that 
  $$\#\{Q'\in\qp:Q\cap Q'\not=\emptyset\}\leq n_0$$ 
  for all $Q\in\qp.$ 
\item[ii.]
 Let $0\leq
  \alpha\leq 1$. An admissible covering is called an (anisotropic) $\alpha$-covering of $\bR^d$ if
$|Q|^{1/\nu}\asymp \brac{\xi}^\alpha$ (uniformly) for all $x\in Q$
and for all $Q\in\qp$, and there exists 
a constant $K\geq 1$ such that $R_Q/r_Q\leq K$ for
all $Q\in \qp$.
\end{itemize}
\end{mydef}

We have already proven in Lemma \ref{le:cut}.(d) that for $\alpha\in [0,1)$, the covering $\{\mathcal{K}_j\}_{j=0}^\infty$ considered in Eq.\ \eqref{eq:Kj} satisfies, for $R_k^j\in \mathcal{K}_j$,
\begin{equation}\label{eq:Qrule}
|R^j_k|^{1/\nu}\asymp \brac{\cx}^\alpha,\qquad\text{(uniformly) for all }\cx\in R_k^j,
\end{equation}
provided we put $\beta=(1-\alpha)^{-1}$.
The inscribed/circumscribed ball ratio $R_{R_k^j}/r_{R_k^j}$  is also  uniformly bounded as shown in (c) of Lemma \ref{le:cut}. We  conclude that the family  $\{\mathcal{K}_j\}_{j=0}^\infty$ is indeed an (anisotropic) $\alpha$-covering. 

In order to define smoothness spaces adapted to $\alpha$-coverings, we need to consider an associated slightly expanded $\alpha$-covering defined by selecting a constant $c_7>\max\{1,c_6\}$, with $c_6$ defined in Lemma \ref{le:cut}.(c). We then consider the rectangles
$$\tilde{R}_k^j:=T_k^j(c_7[-1,1]^d),\qquad j\geq 0, k=1,\ldots, |\mathcal{K}_j|,$$
where $T_k^j$ is the affine map defined in \eqref{eq:affine}. It is straightforward to verify that $\{\tilde{R}_k^j\}_{j,k}$ also satisfies Definition \ref{def:1}.
We now take any $\Psi\in C^\infty(\bR^d)$ satisfying $\Psi(\cx)=1$ for $\cx\in c_6[-1,1]^d$ and $\text{supp}(\Psi)\subseteq c_7[-1,1]^d$. For $R_k^j\in\mathcal{K}_j$, we define
\begin{equation}\label{eq:bap}
\Psi_{R_k^j}(\cdot):=\Psi((T_k^j)^{-1}\cdot),\qquad 
\phi_{R_k^j}(\cdot):=\frac{\Phi_{R_k^j}(\cdot)}{\sum_{\ell=0}^\infty \sum_{m=1}^{|\mathcal{K}_\ell|} \Phi_{R_m^\ell}(\cdot)}.
\end{equation}
By construction,
$$\sum_{j,k}\phi_{R_k^j}(\xi)=  1, \qquad \xi\in \bR^d,$$
with the sum being (uniformly) locally finite by property (i) in Definition \ref{def:1}. By following the approach outlined in \cite[Section 4]{Borup2008}, one can verify that $\sup_{j,k}\|\phi_{R_k^j}\|_{H^s_2}<\infty$ for any $s>0$, with the Sobolev norm $\|\cdot\|_{H^s_2}$ defined in Eq.\ \eqref{eq:sob}, making $\Phi:=\{\phi_{R_k^j}\}_{j,k}$ a so-called Bounded Admissible Partition of Unity (BAPU) associated with the $\alpha$-covering $\{\tilde{R}_k^j\}_{j,k}$, we refer to  \cite{Borup2008} for further details.


We can now define the  (anisotropic) Triebel-Lizorkin type spaces and the decomposition spaces. We let $ \mathcal{S}':=\mathcal{S}'(\bR^d)$ denote the class of tempered distributions defined on $\mathbb{R}^d$.

\begin{mydef}\label{def:complete}
Let $0\leq \alpha<1$ and let $\Psi$ be the BAPU introduced in \eqref{eq:bap} associated with the $\alpha$-covering  $\{\tilde{R}_k^j\}_{j,k}$.  For each $\phi_{R_k^j} \in\Phi$,  we put $\phi_{R_k^j}(D)f := \mathcal{F}^{-1}(\phi_{R_k^j}\mathcal{F}f)$.
\begin{itemize}
	\item For $s\in\R, 0<p<\infty$ and $0<q\leq \infty$, we define the (anisotropic) $\alpha$-Triebel-Lizorkin space $\dot{F}_{p,q}^{s,\alpha}(\va)$ as the set all $f\in \mathcal{S}'$ satisfying
	\begin{equation*}
	\norm{f}_{{F}_{p,q}^{s,\alpha}(\va)} := \norm{\left( \sum_{j=0}^\infty \sum_{k=1}^{|\mathcal{K}_j|}\big||R_k^j|^{s/\nu}\phi_{R_k^j}(D)f\big|^q\right)^{1/q}}_{L_p(\bR^d)} < \infty.
	\end{equation*}
	\item For $s \in \mathbb{R}, 0<p\leq \infty$ and $0<q\leq\infty$ we define the (anisotropic) $\alpha$-modulation space ${M}_{p,q}^{s,\alpha}(\va)$ as the set of all $f\in\mathcal{S}'$ satisfying
	\begin{equation*}
	\norm{f}_{{M}_{p,q}^{s,\alpha}(\va)}  :=
	\left( \sum_{j=0}^\infty \sum_{k=1}^{|\mathcal{K}_j|} |R_k^j|^{s/\nu} 
    \norm{\phi_{R_k^j}(D)f}_{L_p(\bR^d)}^q \right)^{1/q} < \infty,
	\end{equation*}
	with the modification that the summation is replaced by $\sup_{j,k}$ when $q=\infty$.
\end{itemize}
\end{mydef}

\begin{myre}
 The expressions $\norm{\cdot}_{{F}_{p,q}^{s,\alpha}(\va)}$ and $\norm{\cdot}_{{M}_{p,q}^{s,\alpha}(\va)}$ in Definition \ref{def:complete} depend on $\alpha$  through the geometry of the rectangles $\{R_k^j\}$.  To make the $\alpha$-dependence more explicit, we may consider the center points $\{\cc_k^j\}_{j,k}$ for the $\alpha$-covering $\{{R}_k^j\}_{j,k}$ as defined in Eq.\ \eqref{eq:affine}. By \eqref{eq:Qrule}, we have, uniformly,
   $ \brac{\cc_k^j}^\alpha\asymp |{R}^j_k|^{1/\nu}$, so we may clearly replace 
  $|{R}^j_k|^{s/\nu}$ by  $\brac{\cc_k^j}^{\alpha s}$ in Definition \ref{def:complete} to obtain equivalent expressions for $\norm{\cdot}_{{F}_{p,q}^{s,\alpha}(\va)}$ and $\norm{\cdot}_{{M}_{p,q}^{s,\alpha}(\va)}$.
\end{myre}

It can be verified, see \cite[Proposition 5.2]{Borup2008}, that ${F}_{p,q}^{s,\alpha}(\bR^d)$ and  ${M}_{p,q}^{s,\alpha}(\bR^d)$ are quasi-Banach spaces if  $0<p<1$ or $0<q<1$, and they are Banach spaces when $1\leq p,q <\infty$. Let $\mathcal{S}:=\mathcal{S}(\bR^d)$ denote the Schwartz space, then we have the embeddings 
$$\mathcal{S}\hookrightarrow {M}_{p,q}^{s,\alpha}(\va)\hookrightarrow
\cS',\qquad \cS\hookrightarrow {F}_{p,q}^{s,\alpha}(\va)\hookrightarrow
\cS',$$ see \cite{AlJawahri2019, Borup2008}.  Moreover, if $p,q<\infty$, $\mathcal{S}$ is dense in both
$M^{s,\alpha}_{p,q}(\va)$ and $F^{s,\alpha}_{p,q}(\va)$.
 The particular space does not (up to norm equivalence)  depend on the choice of BAPU nor does it depend on the particular choice of sample frequencies $\cc_k^j$ as long as $\cc_k^j\in \tilde{R}_k^j$,  see \cite[Proposition 5.3]{Borup2008}. 

To obtain  further embeddings of $M^{s,\alpha}_{p,q}(\va)$ and $F^{s,\alpha}_{p,q}(\va)$ it is convenient to use the ordering on the countable set $\mathcal{K}$ given in \eqref{eq:K} and express
the norms of $M^{s,\alpha}_{p,q}(\va)$ and $F^{s,\alpha}_{p,q}(\va)$
using the $L_p(\ell_q)$ and $\ell_q(L_p)$-norms, defined for  a sequence
$f=\{f_j\}_{j\in\bN}$ of measurable functions by
\begin{equation}\label{eq:lplq}
\|f\|_{L_p(\ell_q)}:=\bigg\|\bigg(\sum_{j\in\bN} 
|f_j|^q\bigg)^{1/q}\bigg\|_{L_p(\bR^d)},\qquad \|f\|_{\ell_q(L_p)}=\bigg(\sum_{j\in\bN} 
\|f_j\|_{L_p(\bR^d)}^q\bigg)^{1/q},
\end{equation}
where $0<p,q< \infty$. 
In fact, for  $f\in \mathcal{S}'$, it follows directly from Definition \ref{def:complete} that
$$\norm{f}_{{F}_{p,q}^{s,\alpha}(\va)} = \bigg\|\Big\{|R|^{s/\nu}\phi_{R}(D)f\Big\}_{R\in\mathcal{K}}\bigg\|_{L_p(\ell_q)},$$
and
$$\norm{f}_{{M}_{p,q}^{s,\alpha}(\va)} = \bigg\|\Big\{|R|^{s/\nu}\phi_{R}(D)f\Big\}_{R\in\mathcal{K}}\bigg\|_{\ell_q(L_p)} .$$
Using this identification, we can apply general estimates, see  \cite[\S 2.3.2]{Triebel1983}, on $L_p(\ell_q)$ and $\ell_q(L_p)$-norms to obtain the general embeddings 

\begin{equation}\label{eq:emb}
{M}_{p,\min\{p,q\}}^{s,\alpha}(\va)\hookrightarrow {F}_{p,q}^{s,\alpha}(\va)
\hookrightarrow {M}_{p,\max\{p,q\}}^{s,\alpha}(\va),
\end{equation}
valid for $0<p,q<\infty$ and $s\in\bR$.

\subsection{Characterizations of ${F}_{p,q}^{s,\alpha}(\va)$ and ${M}_{p,q}^{s,\alpha}(\va)$}
We claim that the spaces ${F}_{p,q}^{s,\alpha}(\va)$ $[{M}_{p,q}^{s,\alpha}(\va)]$ can be completely characterized for $\alpha \in [0,1)$ using the full orthonormal brushlet system $\mathcal{W}_\beta$ considered in Proposition \ref{prop:onb}, provided we put $\beta=(1-\alpha)^{-1}$. 

Here we focus on proving this claim for the Triebel-Lizorkin type spaces ${F}_{p,q}^{s,\alpha}(\va)$ spaces. In several of the proofs in this Section, we will need the results on vector-valued multiplies that can be found in Appendix \ref{s:app}. The major technical challenge that will be addressed  is how to handle the specific structure of the tensor brushlet basis, where each atom has $2^d$ ``humps'' in the time domain due to the structure of the univariate brushlets, c.f., Eqs.\ \eqref{eq:gw} and \eqref{eq:G}. The $\alpha$-modulation spaces ${M}_{p,q}^{s,\alpha}(\va)$ are easier to handle due to their (simpler) structure, and we leave the the adaptation of the proofs to this case for the reader.

 Let us first study the canonical coefficient operator for the tensor brushlet basis.
 We use the notation introduced in Eq.\ \eqref{eq:Kj}. Let $j\geq 0$ and let $R_k^j\in\mathcal{K}_j$. 
Write $$R_k^j=I_1\times I_2\times\cdots \times I_d$$ and put $\delta_{R_k^j}=\text{diag}(|I_1|,|I_2|,\ldots,|I_d|)$. We define
\begin{equation}\label{eq:en}
\ce_{R_k^j,\cn}:= \pi\delta_{R_k^j}^{-1}  \bigl( \cn+{\textstyle
    \ca}\bigr),\qquad
 n\in \N_0^d,
 \end{equation}
where $\ca:=[\frac12,\ldots,\frac12]^T\in\bR^d$. We put
\begin{equation}
  \label{eq:QD}
  U(R^j_k,\cn)=\left\{\cy\in\bR^d:
\delta_{R_k^j}\cy-\pi{\left(\cn+\ca\right)}\in B_{\vec{a}}(0,1)\right\}.
\end{equation}
We notice that clearly $\ce_{R_k^j,\cn}\in U(R^j_k,\cn)$, and using  Lemma \ref{le:cut}.(b), 
 we have 
$|U(R^j_k,\cn)|= |R^j_k|^{-1}\asymp j^{\nu(1-\beta)}$ uniformly in $j$ and $k$.
One may verify directly from \eqref{eq:QD} that $\cup_{\cn}U(R^j_k,\cn)= \bR^d$, and there exists $L<\infty$ so that uniformly in $\cx$
and $R_k^j$,  \begin{equation}\label{eq:L1}
    \sum_{\cn\in\bN_0^d}\mathds{1}_{U(R^j_k,\cn)}(\cx)\leq L,
\end{equation}
with $\mathds{1}_A$ denoting the characteristic function of a measurable set $A$.
  We also deduce from \eqref{eq:QD}, the estimate from Lemma \ref{le:cut}.(a), and Eq.\ \eqref{eq:equiv_norm} that there exists a  constant  $K$ such that, for any $j$, $k$,
\begin{equation}\label{eq:L2}
U(R_k^j,\cn)\cap U(R_k^j,\cn')\not=\emptyset \Longrightarrow
|\cu-\cv|_{\vec{a}}\leq K j^{1-\beta},\qquad \cu\in U(R_k^j,\cn),\cv\in U(R_k^j,\cn').
\end{equation}

We can now prove that the canonical coefficient operator is bounded on
$F^s_{\vp,q}(\va)$ in the following sense.

\begin{myprop}\label{prop:normchar}
Suppose $s\in\bR$, $0<p< \infty$, and $0<q< \infty$. Then
$$\|\mathcal{S}_q^s(f)\|_{L_p(\bR^d)}\leq C\|f\|_{{F}^{s,\alpha}_{p,q}(\va)},\qquad f\in {F}^{s,\alpha}_{\vp,q}(\va),$$
where
\begin{equation}
    \label{eq:Sq}
\mathcal{S}_q^s(f):=\Big(\sum_{j=0}^\infty\sum_{R \in \mathcal{K}_j} \sum_{\cn\in \bN_0^d} \big(|R|^{s/\nu}|\langle
f,w_{R,\cn}\rangle|\widetilde{\mathds{1}}_{U(R,\cn)}(\cdot)\big)^q \Big)^{1/q},
    \end{equation}
with $\widetilde{\mathds{1}}_{U(R,n)}:=|R|^{1/2}\mathds{1}_{U(R,n)}$. 
\end{myprop}

\begin{proof}
Take $f\in  \mathcal{S}(\bR^d)\subset {F}^{s,\alpha}_{p,q}(\va)$ and consider   $R_k^j\in \mathcal{K}_{j}$,  for some $j\in \bN$, i.e.,  $|R_k^j|\asymp j^{\nu(\beta -1)}$.
Let  $T_k^i(\cdot):=j^{(\beta-1)\vec{a}} \cdot+\cc_k^j$ be defined as in Eq.\ \eqref{eq:affine}, and let $\ce_{n,R_j^k}$ be defined as in Eq.\ \eqref{eq:en}.

We define the diagonal matrices
\begin{equation}\label{eq:umm}
    O_m:=\text{diag}(\cv_m),
\end{equation}
with $\cv_m\in\R^d$ chosen such that $\cup_{m=1}^{2^d}\cv_m= \{-1,1\}^d$.

We write the cosine term in Eq.\ \eqref{eq:brush1} as a sum of complex exponentials, and we take a tensor product to create $w_{\cn,R}$ . This process creates a multivariate function with $2^d$ "humps", and, as it turns out, we will consequently need $2^d$ terms to control the inner product $\langle
f,w_{R_j^k,\cn}\rangle$. By \eqref{eq:brush1} and \eqref{eq:S},
\begin{align}
|\langle
f,w_{R_j^k,\cn}\rangle|&\leq \frac{2^{d/2}}{|R_j^k|^{1/2}}\sum_{m=1}^{2^d}|(b_{R_j^k}(D)f)(O_m \ce_{R_j^k,\cn})|\nonumber\\
&\leq c_{d,q} \frac{2^{d/2}}{|R_j^k|^{1/2}}\Big(\sum_{m=1}^{2^d}|(b_{R_j^k}(D)f)(O_m \ce_{R_j^k,\cn})|^q\Big)^{1/q}.\label{eq:innerp}
\end{align}
We notice by estimate (c) in Lemma  \ref{le:cut} there exists a constant $c_7>0$, independent of $j$ and $k$, such that
\begin{equation}\label{eq:structured}
    \text{supp}(b_{R_k^j})\subseteq T_k^j(c_7[-1,1]^d).
\end{equation}
Hence, for $0<t<\min(p,q)$, and $\cx\in\bR^d$, it follows from \eqref{eq:innerp} that
\begin{align}
\sum_{\cn\in \bN_0^d}& \big(|{R_j^k}|^{s/\nu}|\langle
f,w_{{R_j^k},\cn}\rangle|\widetilde{\mathds{1}}_{U({R_j^k},\cn)}(\cx)\big)^q\nonumber\\
&\leq c_{d,q}^q\sum_{\cn\in \bN_0^d}\sum_{m=1}^{2^d}|{R_j^k}|^{sq/\nu} |(b_{R_j^k}(D)f)(O_m \ce_{\cn,R_j^k})|^q\mathds{1}_{U({R_j^k},\cn)}(\cx)
\nonumber
\\
&\le  c_{d,q}^q\sum_{m=1}^{2^d}\sum_{\cn\in \bN_0^d}
|{R_j^k}|^{sq/\nu}\bigg(\sup_{\cy\in U(R_j^k,\cn)}|(b_{R_j^k}(D)f)(O_m \cy)|\bigg)^q\mathds{1}_{U({R_j^k},\cn)}(\cx)
\nonumber
\\
&\le c_L\sum_{m=1}^{2^d}
|{R_j^k}|^{sq/\nu}\bigg(\sup_{\cz\in B_{\vec{a}}(0,Kj^{1-\beta})}|(b_{R_j^k}(D)f)(O_m (\cx-\cz)| \langle j^{(\beta-1)\vec{a}}O_m\cz\rangle_{\vec{a}}^{-\nu/t}\nonumber\\
&\hspace{7cm}\times\langle j^{(\beta-1)\vec{a}}O_m\cz\rangle_{\vec{a}}^{\nu /t}\bigg)^q\nonumber\\
&\le c_{L,K}\sum_{m=1}^{2^d}\bigg(|{R_j^k}|^{sq/\nu}\sup_{\cw\in \bR^d}\frac{|(b_{R_j^k}(D)f)(O_m\cx-\cw)| }{\brac{j^{(\beta-1)\vec{a}}\cw}^{\nu/t}}\bigg)^q\nonumber\\
&= c_{L,K}\sum_{m=1}^{2^d}\bigg(|{R_j^k}|^{s/\nu}\sup_{\cv\in \bR^d}\frac{|(b_{R_j^k}(D)f)(\cv)|}{\brac{j^{(\beta-1)\vec{a}}(O_m\cx-\cv)}^{\nu/t}}\bigg)^q,
\label{Thmainp6}
\end{align}
where we have used \eqref{eq:L1}, \eqref{eq:L2}, and the fact   that $B_{\vec{a}}(0,Kj^{1-\beta})$ is invariant under the transformations $\cy\rightarrow O_m \cy$. We now continue the estimate \eqref{Thmainp6} by calling on the Peetre maximal-function estimate \eqref{M3} to obtain
\begin{align}
\sum_{\cn\in \bN_0^d} (|{R_j^k}|^{s/\nu}|\langle
f,w_{\cn,{R_j^k}}\rangle|\widetilde{\mathds{1}}_{U({R_j^k},\cn)}(\cx))^q
&\leq c'\sum_{m=1}^{2^d}\bigg(|{R_j^k}|^{s/\nu}\mathcal{M}_t\big((b_{R_j^k}(D)f\big)(O_m(\cx))\bigg)^q.
\label{Thmainp88}
\end{align}
Hence,
\begin{align}
\|\mathcal{S}_q^s(f)\|_{L_p(\bR^d)}&\leq C_a
\bigg\|\bigg(\sum_{m=1}^{2^d} \sum_{j=0}^\infty\sum_{R_j^k \in \mathcal{K}_j}\big(|{R_j^k}|^{s/\nu}\mathcal{M}_t\big((b_{R_j^k}(D)f\big)(O_m(\cdot))\big)^q\bigg)^{1/q}\bigg\|_{L_p(\bR^d)}\nonumber\\
&\leq C_b\bigg\|\bigg(\sum_{m=1}^{2^d} \sum_{j=0}^\infty\sum_{R_k^j \in \mathcal{K}_j}\big(|{R_j^k}|^{s/\nu}\big|b_{R_j^k}(D)f(O_m(\cdot))\big|\big)^q\bigg)^{1/q}\bigg\|_{L_p(\bR^d)}\nonumber\\
&\leq C_c\sum_{m=1}^{2^d}\bigg\|\bigg( \sum_{j=0}^\infty\sum_{R_k^j \in \mathcal{K}_j}\big(|{R_j^k}|^{s/\nu}\big|b_{R_j^k}(D)f(O_m(\cdot))\big|\big)^q\bigg)^{1/q}\bigg\|_{L_p(\bR^d)}\nonumber\\
&\leq C_d\bigg\|\bigg( \sum_{j=0}^\infty\sum_{R_k^j \in \mathcal{K}_j}\big(|{R_j^k}|^{s/\nu}\big|b_{R_j^k}(D)f(\cdot)\big|\big)^q\bigg)^{1/q}\bigg\|_{L_p(\bR^d)},
\label{Thmainp99}
\end{align}
where we used the Fefferman-Stein maximal inequality, see \eqref{eq:fs}, the (quasi-)triangle inequality on $L_p(\ell_q)$, and $2^d$ substitutions of the form $\cy=O_m(\cx)$ in the integrals.

Guided by the observation in Eq.\ \eqref{eq:structured}, we now wish to apply Theorem \ref{th:vecmult} with the multipliers $\{b_{R_k^j}\}_{j,k}$ in \eqref{Thmainp99}.  We first notice that each multiplier $b_{R_j^k}(T_k^j\,\cdot\,)$ is smooth with 
$$\text{supp}\big\{b_{R_j^k}(T_k^j\,\cdot\,)\big\}\subseteq c_7[-1,1]^d,$$
and from Eq.\ \eqref{eq:GG}, we have the relation
$$b_{R_j^k}(T_k^j\,\cdot\,)=\widehat{G_{R_k^j}}\big(\delta_{R_k^j}^{-1}(j^{(\beta-1)\vec{a}} \cdot+\cc_k^j-\bm{\alpha}_{R_k^j})\big),$$
where $\widehat{G_{R_k^j}}$ is the centralized bell-function defined in Eq.\ \eqref{eq:Gf}. By Lemma \ref{le:cut}.(a), we deduce that the matrix representation of $\delta_{R_k^j}^{-1}j^{(\beta-1)\vec{a}}$ satisfies
$$\|\delta_{R_k^j}^{-1}j^{(\beta-1)\vec{a}}\|_{\ell_\infty(\bR^{d\times d})}\leq c<\infty,$$
for some $c$ independent of $j$ and $k$. We may therefore use the chain-rule, and the properties from Eqs.\ \eqref{eq:gb} and \eqref{eq:epsilon}, to conclude that there exists constants $K_\beta$, $\beta\in\bN_0^d$, independent of $j$ and $k$, such that 
$$\big\|\partial^\beta\big[b_{R_j^k}(T_k^j\,\cdot\,)\big]\big\|_{L_\infty(\bR^d)}\leq K_\beta.$$
We then obtain the estimate, for any $N\in\bN$,
\begin{align*}
|[\mathcal{F}^{-1}b_{R_j^k}(T_k^j\,\cdot\,)](\cx)|&\leq c(1+|\cx|)^{-N}\bigg| \sum_{\beta\in\bN_0^d: |\beta|\leq N} x^\beta [\mathcal{F}^{-1}b_{R_j^k}(T_k^j\,\cdot\,)](x)\bigg|\\
&\leq c' (1+|\cx|)^{-N} \sum_{\beta\in\bN_0^d: |\beta|\leq N} \|\partial^\beta\big[b_{R_j^k}(T_k^j\,\cdot\,)\big]\|_{L_1(\bR^d)}\\
&\leq c''\bigg(\sum_{\beta\in\bN_0^d: |\beta|\leq N} K_\beta\bigg) (1+|\cx|)^{-N},\qquad \cx\in\bR^d.
\end{align*}
Using the estimate \eqref{ad8}, it now follows easily that the hypotheses of Theorem \eqref{th:vecmult} are satisfied  for the multipliers $\{b_{R_k^j}\}_{j,k}$. 
Now notice that $$b_{R_k^j}(D)f=b_{R_k^j}(D)\sum_{(\ell,m)\in F_{R_k^j}} \phi_{R_m^\ell}(D) f,$$
where 
$$F_{R_k^j}:=\{(\ell,m): \tilde{R}_k^j\cap \tilde{R}_m^\ell\not=\emptyset\}.$$
Since we have an $\alpha$-covering, there is a universal constant $M>0$ such that $\# F_{R_k^j}\leq M$, and, with constants independent of $j$ and $k$,
$|{R}_k^j|\asymp |{R}_m^\ell|$ for $(\ell,m) \in F_{R_k^j}$ by (d) of Lemma \ref{le:cut}.
We now continue estimate \eqref{Thmainp99}  and obtain
\begin{align*}
\|\mathcal{S}_q^s(f)\|_{L_p(\bR^d)}&
\leq C_d\bigg\|\bigg( \sum_{j=0}^\infty\sum_{R_k^j \in \mathcal{K}_j}\Big(|{R_j^k}|^{s/\nu}\big|b_{R_j^k}(D)f(\cdot)\big|\Big)^q\bigg)^{1/q}\bigg\|_{L_p(\bR^d)}\\
&
\leq C_d\bigg\|\bigg( \sum_{j=0}^\infty\sum_{R_k^j \in \mathcal{K}_j}\Big(|{R_j^k}|^{s/\nu}\Big|b_{R_k^j}(D)\sum_{(\ell,m)\in F_{R_k^j}}\phi_{R_m^\ell}(D)f\Big|\Big)^q\bigg)^{1/q}\bigg\|_{L_p(\bR^d)}\\
&
\leq C_d'\bigg\|\bigg( \sum_{j=0}^\infty\sum_{R_k^j \in \mathcal{K}_j}\sum_{(\ell,m)\in F_{R_k^j}}\Big(|{R_j^k}|^{s/\nu}\Big|b_{R_k^j}(D)\phi_{R_m^\ell}(D)f\Big|\Big)^q\bigg)^{1/q}\bigg\|_{L_p(\bR^d)}\\
&\leq C_d''\bigg\|\bigg( \sum_{j=0}^\infty\sum_{R_k^j \in \mathcal{K}_j}\sum_{(\ell,m)\in F_{R_k^j}}\Big(|{R_m^\ell}|^{s/\nu}\Big|\phi_{R_m^\ell}(D)f\Big|\Big)^q\bigg)^{1/q}\bigg\|_{L_p(\bR^d)}\\
&= C_d''\bigg\|\bigg( \sum_{j=0}^\infty\sum_{R_k^j \in \mathcal{K}_j} \sum_{\ell=0}^\infty\sum_{R_m^\ell \in \mathcal{K}_\ell}\mathds{1}_{F_k^j}(\ell,m)\Big(|{R_m^\ell}|^{s/\nu}\Big|\phi_{R_m^\ell}(D)f\Big|\Big)^q\bigg)^{1/q}\bigg\|_{L_p(\bR^d)}\\
&= C_d''\bigg\|\bigg(  \sum_{\ell=0}^\infty\sum_{R_m^\ell \in \mathcal{K}_\ell}\bigg[\sum_{j=0}^\infty\sum_{R_k^j \in \mathcal{K}_j}\mathds{1}_{F_k^j}(\ell,m)\bigg]\Big(|{R_m^\ell}|^{s/\nu}\Big|\phi_{R_m^\ell}(D)f\Big|\Big)^q\bigg)^{1/q}\bigg\|_{L_p(\bR^d)}\\
&\leq C_d'''\bigg\|\bigg(  \sum_{\ell=0}^\infty\sum_{R_m^\ell \in \mathcal{K}_\ell}\Big(|{R_m^\ell}|^{s/\nu}\Big|\phi_{R_m^\ell}(D)f\Big|\Big)^q\bigg)^{1/q}\bigg\|_{L_p(\bR^d)}\\
&\asymp
\|f\|_{{F}^{s,\alpha}_{p,q}(\va)},
\end{align*}
where we used Theorem \ref{th:vecmult}, Tonelli's theorem to change the summation order, and the uniform bound on the cardinality of the sets $F_m^\ell$. Finally, to conclude the proof, we simply extend the estimate to all of $ {F}^{s,\alpha}_{p,q}(\va)$, using the fact that  $\mathcal{S}(\R^d)$ is  dense in $ {F}^{s,\alpha}_{p,q}(\va)$. 
\end{proof}

\begin{myre}
The proof of Proposition \ref{prop:normchar} can be simplified considerably provided it is possible to replace the cosines in \eqref{eq:brush1} by localized exponential functions in the frequency domain as building blocks for the tensor basis construction. However, the author is not aware of any such construction of orthonormal bases that are also well-localized. Moreover, the Balian-Low Theorem, see \cite[Theorem 8.4.1]{Groechenig2001}, provides strong evidence that such constructions are not possible in general.  
\end{myre}

Inspired by Proposition \ref{prop:normchar}, we define the sequence space
${f}^{s,\alpha}_{p,q}(\vec{a})$ for $s\in \bR$, $0<p<\infty$, and
$0<q\leq \infty$ , as the set of sequences 
$\{s_{R,\cn}\}_{R\in \mathcal{K},\cn\in\bN_0^d}\subset \mathbb{C}$ satisfying
\begin{equation}\label{eq:f}
\|\{s_{R,\cn}\}\|_{{f}^{s,\alpha}_{p,q}(\vec{a})}:= \bigg\| 
\Big\{ \sum_{\cn\in \bN_0^d} |R|^{s/\nu}|s_{R,\cn}|\widetilde{\mathds{1}}_{U(R,\cn)}(\cdot) \Big\}_{R\in \mathcal{K}}
\bigg\|_{L_p(\ell_q)}<\infty,
\end{equation}
where we let $\widetilde{\mathds{1}}_{U(R,n)}:=|R|^{1/2}\mathds{1}_{U(R,n)}$ and we use the ordering on the countable set $\mathcal{K}$ given in \eqref{eq:K}.   
In a similar fashion, we we define the sequence space
${m}^{s,\alpha}_{p,q}(\vec{a})$ for $s\in \bR$, $0<p<\infty$, and
$0<q\leq \infty$ , as the set of sequences 
$\{s_{R,\cn}\}_{R\in \mathcal{K},\cn\in\bN_0^d}\subset \mathbb{C}$ satisfying
\begin{equation}\label{eq:m}
\|\{s_{R,\cn}\}\|_{{m}^{s,\alpha}_{p,q}(\vec{a})}:= \bigg\| 
\Big\{ \sum_{\cn\in \bN_0^d} |R|^{s/\nu}|s_{R,\cn}|\widetilde{\mathds{1}}_{U(R,\cn)}(\cdot) \Big\}_{R\in\mathcal{K}}
\bigg\|_{\ell_q(L_p)}<\infty.
\end{equation}
The (quasi-)norm on ${m}^{s,\alpha}_{p,q}(\vec{a})$ can be expressed as a discrete mixed-norm by noticing that the sum over $\cn$ is uniformly finite by  \eqref{eq:L1}, so by the equivalence of quasi-norms on $\bR^L$,
\begin{equation}\label{eq:locallyf}
\bigg(\sum_{\cn\in \bN_0^d} |R|^{s/\nu}|s_{R,\cn}|\widetilde{\mathds{1}}_{U(R,\cn)}(\cdot)\bigg)^p\asymp
\sum_{\cn\in \bN_0^d} |R|^{\frac{sp}{\nu}+\frac{p}2}|s_{R,\cn}|^p\mathds{1}_{U(R,\cn)}(\cdot).
\end{equation}
Hence, using that $\|\mathds{1}_{U(R,\cn)}\|_{L^p}\asymp |R|^{-1/p}$,
\begin{equation}\label{eq:malt}
\|\{s_{R,\cn}\}\|_{{m}^{s,\alpha}_{p,q}(\vec{a})}\asymp \bigg(\sum_{R\in\mathcal{K}}|R|^{(\frac{s}{\nu}+\frac{1}2-\frac{1}p)q} \bigg(\sum_{\cn\in\bN_0^d} |s_{R,\cn}|^p\bigg)^{q/p}\bigg)^{1/q}.
\end{equation}
We have the following estimate.

\begin{myprop}\label{lem:reconstruct}
  Suppose  $s\in \bR$, $0<p< \infty$, and $0<q\leq \infty$. Then for
  any finite sequence
  $\{s_{R,\cn}\}_{(R,\cn)\in\mathcal{F}}$, with $\mathcal{F}\subset \mathcal{K}\times \bN_0^d$, we have 
  $$\Bigl\|\sum_{(R,\cn)\in\mathcal{F}} s_{R,\cn}w_{R,\cn}\Bigr\|_{\dot{F}_{p,q}^{s,\alpha}(\va)} \leq C
  \|\{s_{R,\cn}\}\|_{{f}^{s,\alpha}_{p,q}(\va)}.$$
\end{myprop}

\begin{proof}
Let $\{\phi_{R_k^j}\}_{j,k}$ be the bounded partition of unity defined in 
\eqref{eq:bap}. Using the structure given by \eqref{eq:brush1}, and Proposition \ref{th:vecmult}, we get
  \begin{align*}
    \Bigl\|\sum_{(R,\cn)\in\mathcal{F}} s_{R,\cn}w_{R,\cn}\Bigr\|_{\dot{F}_{p,q}^{s,\alpha}(\va)} &= \Bigl\|
    \Big\{|R_k^j|^{s/\nu} \phi_{R_k^j}(D)\Big(\sum_{(R,\cn)\in\mathcal{F}} 
    s_{R,\cn}w_{R,\cn}\Big)\Big\}_{j,k}\Bigr\|_{L_{{p}}(\ell_q)}\\
    &\leq C\Bigl\|\Big\{
    |R_k^j|^{s/\nu} \sum_{R \in N_{j,k}} \sum_{\cn\in\bN_0^d}
    s_{R,\cn}w_{R,\cn}\Big\}_{j,k}\Bigr\|_{L_{{p}}(\ell_q)},
\end{align*}
where $N_{j,k}= \{R'\in \mathcal{K}\colon \tilde{R}_k^j\cap \supp
(b_{R'})\neq \emptyset\}$. It follows easily from the properties of $\alpha$-coverings  that $\# N_{j,k}$ is
uniformly bounded and that $|R'|\asymp |R_k^j|$ (uniformly) for $R'\in N_{j,k}$. Hence, we have 
\begin{equation}\Bigl\|
   \Big\{
    |R_k^j|^{s/\nu} \sum_{R \in N_{j,k}} \sum_{\cn\in\bN_0^d}
    s_{R,\cn}w_{R,\cn}\Big\}_{j,k}\Bigr\|_{L_{{p}}(\ell_q)} \leq
    C\biggl\|\biggl( \sum_{R\in \mathcal{K}} \Bigl( |R|^{s/\nu}  \sum_{\cn\in\bN_0^d}
    |s_{R,\cn}| |w_{R,\cn}|\Bigr)^q\biggr)^{1/q}\biggr\|_{{p}}.
    \end{equation}
We fix $0<r<\min(p,q)$. To simplify the notation, we use the ordering on $\mathcal{K}$ from \eqref{eq:K}. Then Lemma \ref{lem:maxbound} and the
Fefferman-Stein maximal inequality \eqref{eq:fs} yield
\begin{align*}
    \Bigl\|\Big\{
    |R|^{s/\nu} & \sum_{\cn\in\bN_0^d} |s_{R,\cn}| |w_{R,\cn}|\Big\}_{R\in\mathcal{K}}
\Bigr\|_{L_{{p}}(\ell_q)}\\
    &\leq C_d\Bigl\|\Big\{
    |R|^{s/\nu} |R|^{1/2} \sum_{m=1}^{2^d} \mathcal{M}_r\Bigl(\sum_{\cn\in\bN_0^d}
  |s_{R,\cn}|
  \chi_{U(R,\cn)}\Bigr)(O_{m}\cdot)\Big\}_{R\in\mathcal{K}}
\Bigr\|_{L_{{p}}(\ell_q)}\\
&\leq C_d'\Bigl\|\Big\{
    |R|^{s/\nu} |R|^{1/2} \sum_{\cn\in\bN_0^d}
  |s_{R,\cn}| \mathds{1}_{U(R,\cn)}\Big\}_{R\in\mathcal{K}}\Bigr\|_{L_{{p}}(\ell_q)}\\
  &= C_d'\Bigl\|\Big\{
    |R|^{s/\nu} \sum_{\cn\in\bN_0^d}
  |s_{R,\cn}| \widetilde{\mathds{1}}_{U(R,\cn)}\Big\}_{R\in\mathcal{K}}\Bigr\|_{L_{{p}}(\ell_q)},\\
  &=C_d'\|\{s_{R,\cn}\}\|_{{f}^{s,\alpha}_{p,q}(\va)},
  \end{align*}
  where we used the (quasi-)triangle inequality and straightforward substitutions, $O_m\cx\rightarrow \cy$, in the integrals.

\end{proof}

We now use Proposition \ref{prop:normchar} and Proposition
\ref{lem:reconstruct} to obtain the main result of this paper, that for any $\alpha\in[0,1)$, $\mathcal{W}:=\{w_{R,\cn}\}_{R\in\mathcal{K},\cn\in\bN_0^d}$ forms an orthonormal basis that universally 
captures the norm of ${B}^{s,\alpha}_{p, q}(\va)$ and ${F}^{s,\alpha}_{p, q}(\va)$. Moreover, the system forms an unconditional basis for  ${B}^{s,\alpha}_{p, q}(\va)$ and for ${F}^{s,\alpha}_{p, q}(\va)$ in the Banach space case. 
Calling on Proposition \ref{prop:normchar}, we may define 
a bounded coefficient operator
$C_{\mathcal{W}}\colon F^{s,\alpha}_{p,q}(\va)\rightarrow f^{s,\alpha}_{p,q}(\va)$ by
$$C_{\mathcal{W}}f=\{\langle f,w_{R,\cn}\rangle\}_{R\in\mathcal{K},\cn\in\bN_0^d}.$$ 
By Proposition \ref{lem:reconstruct},  
the corresponding reconstruction operator 
$R_{\mathcal{W}}:f^{s,\alpha}_{p,q}(\va)\rightarrow F^{s,\alpha}_{p,q}(\va)$, defined by
$$R_{\mathcal{W}}\big(\{s_{R,\cn}\}_{R,\cn}\big)=\sum_{j=0}^\infty\sum_{k=1}^{|\mathcal{K}_j|}\sum_{\cn\in\bN_0^d} s_{R_k^j,\cn} w_{R_k^j,n},$$
extends to a bounded operator. We summarize these findings in the following theorem.

\begin{mytheorem}\label{prop:modu}
Let $s\in \bR$, $\alpha\in [0,1)$, and  $0<p,q< \infty$. Then we have the bounded operators $C_{\mathcal{W}}\colon F^{s,\alpha}_{p,q}(\va)\rightarrow f^{s,\alpha}_{p,q}(\va)$ and
$R_{\mathcal{W}}:f^{s,\alpha}_{p,q}(\va)\rightarrow F^{s,\alpha}_{p,q}(\va)$, satisfying
$$R_{\mathcal{W}}\circ C_{\mathcal{W}}=\text{Id}_{F^{s,\alpha}_{p,q}(\va)},$$
making $F^{s,\alpha}_{p,q}(\va)$ a retract of $f^{s,\alpha}_{p,q}(\va)$.
Similarly, we have the bounded operators $C_{\mathcal{W}}\colon M^{s,\alpha}_{p,q}(\va)\rightarrow m^{s,\alpha}_{p,q}(\va)$ and
$R_{\mathcal{W}}:m^{s,\alpha}_{p,q}(\va)\rightarrow M^{s,\alpha}_{p,q}(\va)$, satisfying
$$R_{\mathcal{W}}\circ C_{\mathcal{W}}=\text{Id}_{M^{s,\alpha}_{p,q}(\va)},$$
making $M^{s,\alpha}_{p,q}(\va)$ a retract of $m^{s,\alpha}_{p,q}(\va)$.
Consequently,  we have the norm characterizations 
$$ \|f\|_{{F}^{s,\alpha}_{p,q}(\va)}\asymp  \|
C_{\mathcal{W}}f\|_{f^{s,\alpha}_{p,q}(\va)},\qquad f\in {F}^{s,\alpha}_{p,q}(\va),$$
and
$$ \|f\|_{{M}^{s,\alpha}_{p,q}(\va)}\asymp  \|
C_{\mathcal{W}}f\|_{m^{s,\alpha}_{p,q}(\va)},\qquad f\in {M}^{s,\alpha}_{p,q}(\va).$$ 
Moreover, in the Banach space case, i.e., for $1\leq p<\infty$ and $1\leq q<\infty$,  the brushlet system $\mathcal{W}$ forms an unconditional basis for both  ${M}^{s,\alpha}_{p,q}(\va)$ and ${F}^{s,\alpha}_{p,q}(\va)$.
\end{mytheorem}

\begin{proof}
  The boundedness claims about $C_{\mathcal{W}}$ and $R_{\mathcal{W}}$ follows at once from Propositions \ref{prop:normchar} and \ref{lem:reconstruct}, respectively. Also, notice that 
$  R_{\mathcal{W}}\circ C_{\mathcal{W}}$ acts as the identity on the subset $\mathcal{S}(\bR^d)\subset L_2(\bR^d)$ due to the fact that $\mathcal{W}$ is an orthonormal basis for $L_2(\bR^d)$. This identity can then be extended to all of 
  $F^{s,\alpha}_{p,q}(\va)$ using the fact that $\mathcal{S}(\bR^d)$ is dense in $F^{s,\alpha}_{p,q}(\va)$.
  The claim that the system forms an unconditional
basis when $1\leq p, q < \infty$ follows easily from the fact that  $\dot{F}^{s,\alpha}_{p,q}(\va)$ is a Banach space, and
that finite expansions in  $\{w_{R,\cn}\}$ have uniquely determined coefficients giving us a norm characterization of such expansions by the $L_{p}$-norm of the quantity $\mathcal{S}_q^{s}(\cdot)$  defined in Eq.\ \eqref{eq:Sq}.   The proof in the case of $M^{s,\alpha}_{p,q}(\va)$ and $m^{s,\alpha}_{p,q}(\va)$ is similar and is left for the reader.  
\end{proof}

\begin{myre}
The norm characterization obtained in Theorem \ref{prop:modu} may appear similar to the characterization obtained for tight frames in, e.g., \cite[Theorem 7.5]{AlJawahri2019} and \cite[Theorem 6.4]{Borup2008}, but one should notice the important additional  fact that we now have at our disposal, namely that $\{w_{R,\cn}\}$ forms an \textit{orthonormal basis}. The orthogonality ensures that the maps  $C_{\mathcal{W}}$ and $R_{\mathcal{W}}$ considered in Theorem \ref{prop:modu} are in fact {\em isomorphisms}. This fact has significant implications for the applicability of Theorem \ref{prop:modu}. One specific example is to the study of  $m$-term nonlinear approximation with the system $\{w_{R,\cn}\}$, where the linear independence will allow one to obtain inverse estimates of Bernstein type as will be derived in Section \ref{s:approx} below. Inverse estimates are currently out of reach for the redundant frames considered in \cite{Borup2008,AlJawahri2019}, and for redundant wavelet type systems in general, see the discussion of this long-standing open problem in \cite{MR1794807}. 
 \end{myre}
\section{Nonlinear approximation with anisotropic brushlet bases}\label{s:approx}
In this section we will derive several results on  nonlinear $m$-term approximation  with the tensor brushlet system $\mathcal{W}_\beta$ introduced in Proposition \ref{prop:onb}. Let us first introduce some notation and concepts from nonlinear approximation theory.   

Let
$\Dict=\{g_k\}_{k\in\N}$ be a  system in a quasi-Banach space $X$, where we assume that $\text{span}(\Dict)$ is dense in $X$.
We consider the collection of all possible $m$-term expansions with
elements from $\Dict$:
$$\Sigma_m(\Dict):=\Big\{\sum_{i\in \Lambda} c_i g_i\: \Big|\: c_i\in\mathbb{C},
\# \Lambda\leq m\Big\}.$$
The error of the best $m$-term approximation to an element $f\in X$
is then
$$\sigma_m(f,\Dict)_X:=\inf_{f_m\in \Sigma_m(\Dict)} \|f-f_m\|_{X}.$$

We classify functions according to the decay-rate of $\sigma_m$ using the following notion of approximation spaces, see \cite{DeVore1993} for further details.

\begin{mydef}
Let $0\leq \gamma<\infty$ and $0<q\leq \infty$.  The approximation space $\mathcal{A}^\gamma_{\, q}(X,\Dict)$ is defined to be the collection of $f\in X$ for which
  \[|f|_{\mathcal{A}^\gamma_{\, q}(X,\Dict)}:=\bigg(\sum_{m=1}^\infty
  \big(m^\gamma \sigma_m(f,\Dict)_X\big)^q
  \frac{1}{m}\bigg)^{1/q}<\infty.\]
The family $\mathcal{A}^\gamma_{\, q}(X,\Dict)$ is (quasi)normed by
  $\|f\|_{\mathcal{A}^\gamma_{\, q}(X,\Dict)} = \|f\|_X +
  |f|_{\mathcal{A}^\gamma_{\, q}(X,\Dict)}$ for $0<q,\gamma<\infty$, with the
  $\ell_q$ norm replaced by the sup-norm, when $q=\infty$.
\end{mydef}

Suppose $0\leq \alpha<1$ and $0<p,q<\infty$. We let $\mathcal{W}:=\mathcal{W}_{(1-\alpha)^{-1}}=\{w_{R,\cn}\}_{R\in \mathcal{K},\cn\in \N_0^d}$ be the brushlet system considered in Proposition \ref{prop:onb}.
The goal in this section is to obtain quantitative information about $\mathcal{A}^\gamma_{\, q}(X,\mathcal{W})$ for various choices of $X$. We will rely on the fact that $\mathcal{W}$ is a stable non-redundant system supporting the corresponding norm-characterizations from Theorem~\ref{prop:modu}. We mention that the results below therefore cover dimensions $d\geq 2$. Nonlinear approximation with brushlet bases in the simpler case $d=1$ has been considered earlier by Borup and the author in \cite{Borup2006b}.

The fact that $M^{s,\alpha}_{p,q}$ is a retract of $m^{s,\alpha}_{p,q}$ and $F^{s,\alpha}_{p,q}$ is a retract of $f^{s,\alpha}_{p,q}$, as shown in Theorem~\ref{prop:modu}, turns out to be very helpful for this purpose. 
 Consider the sequence $\mathbf{1}_{(R,\cn)}$ defined to have value $1$ at position $(R,\cn)\in \mathcal{K}\times \bN_0^d$ and $0$ elsewhere on $\mathcal{K}\times \bN_0^d$. Clearly, $R_{\cW}({\mathbf{1}}_{(R,\cn)})=w_{R,\cn}$, and 
consequently, for the dictionary of sequences defined by $\Dict_d:=\{\mathbf{1}_{(R,\cn)}\}_{R,\cn}$, 
we  have $R_{\cW}\Sigma_m(\Dict_d)=\Sigma_m(\cW)$ for $m\in\bN$.
It is therefore straightforward to verify that we have the identifications (up to equivalence of norm),
\begin{equation}\label{eq:approxs}
R_{\cW}\mathcal{A}^\gamma_{\, q}(m^{s,\alpha}_{p,q}(\va),\Dict_d)
=\mathcal{A}^\gamma_{\, q}(M^{s,\alpha}_{p,q}(\va),\cW),\qquad
R_{\cW}\mathcal{A}^\gamma_{\, q}(f^{s,\alpha}_{p,q}(\va),\Dict_d)
=\mathcal{A}^\gamma_{\, q}(F^{s,\alpha}_{p,q}(\va),\cW).
\end{equation}

Let us first focus on the modulation-type space. For $0<p<\infty$, we define the corresponding normalized sequences
$$\widetilde{\mathbf{1}}_{(R,\cn)}=\frac{\mathbf{1}_{(R,\cn)}}{\|\mathbf{1}_{(R,\cn)}\|_{m^{s,\alpha}_{p,p}(\va)}},\qquad
R\in \mathcal{K},\cn\in \N_0^d,$$
where we notice directly from \eqref{eq:malt} that (uniformly) $\|\mathbf{1}_{(R,\cn)}\|_{m^{s,\alpha}_{p,p}(\va)}\asymp |R|^{s/\nu+1/2-1/p}$.
We also observe from \eqref{eq:m}  that for any finite sequence $\cs=\{s_{r,\cn}\}$ on $\mathcal{K}\times\N_0^d$,  we have the
uniform estimate
\begin{equation}\label{eq:norm}
\|\cs\|_{m^{s,\alpha}_{p,p}(\va)}=
\bigg\|\sum_{(R,\cn)\in\mathcal{K}\times\N_0^d}
s_{R,\cn}\widetilde{\mathbf{1}}_{(R,\cn)}\bigg\|_{m^{s,\alpha}_{p,p}(\va)}\asymp \bigg(\sum_{R,\cn} |s_{R,\cn}|^p\bigg)^{1/p},
\end{equation}
which shows that $m^{s,\alpha}_{p,p}(\va)$ constitutes a $p$-space in the terminology used in  \cite{MR1247523,Garrigos2004}.

The family of discrete Lorentz spaces for sequences defined on $\mathcal{K}\times\bN_0^d$ will be needed for the analysis of $\mathcal{A}^\gamma_{\, q}(X,\mathcal{W})$. For $0<p,r<\infty$, we let $\ell_{p,r}$ consists of the sequences $\cs=\{s_{R,\cn}\}_{R\in\mathcal{K},\cn\in\bN^d}$ with $\lim s_{R,\cn}=0$, where for  an enumeration $\{I_k\}_{k=1}^\infty$ of $\mathcal{K}\times\bN_0^d$ such that $|s_{I_1}|\geq |s_{I_2}|\geq \cdots$, we have
\begin{equation}\label{eq:lorentz}
    \|\cs\|_{\ell_{p,r}}:=\bigg[\sum_{k=1}^\infty (k^{1/p}|s_{I_k}|)^r\frac{1}{k}\bigg]^{1/r}<+\infty.
\end{equation}
For $r=\infty$, $\ell_{p,\infty}$ is the discrete weak $\ell_p$-space consisting of sequences satisfying 
$$\|\cs\|_{\ell_{p,\infty}}:=\sup_{k\in\bN}  k^{1/p}|s_{I_k}|<+\infty.$$ 
The discrete Lorentz spaces $\ell_{p,r}(\bN)$ for sequences on $\bN$ are defined in a similar way. 

We may use the discrete Lorentz spaces to define a notion of smoothness spaces in order to characterize various
approximation spaces. Let
$\Dict=\{g_k\}_{k\in\N}$ be a {\em quasi-normed} system in $X$ in the sense that $\|g_k\|_X\asymp 1$, $k\in\bN$. For $\tau\in
(0,\infty)$ and $s\in (0,\infty]$, we define the space
$$\K^{\tau}_{s}(X,\Dict):= \{ f\in X\colon f=\sum_{k\in \bN} c_kg_k
\text{  unconditionally in }X,\; g_k\in \Dict,\;
\{c_k\}_{k\in \N}\in \ell_{\tau,s}(\N)\},$$
with $\|f\|_{\K^{\tau}_{\ s}(X,\Dict)} :=
\|\{c_k\}\|_{\ell_{\tau,s}}$. 

For the system 
\begin{equation}\label{eq:tilde}
    \widetilde{\Dict_p}:=\{\widetilde{\mathbf{1}}_{(R,\cn)}\}_{(R,\cn)\in \cK\times\bN_0^d}
\end{equation} in $m^{s,\alpha}_{p,p}(\va)$, it follows directly from the equivalence in \eqref{eq:norm} that $\K^{p}_{p}(m^{s,\alpha}_{p,p}(\va),\Dict_p)=\ell_{p,p}$, and using the embedding 
$\ell_{\tau,q}\hookrightarrow \ell_{p,p}$ for $0<\tau<p$, $0<q\leq \infty$, we may conclude that
\begin{equation}\label{eq:Kdisk}
\K^{\tau}_{q}(m^{s,\alpha}_{p,p}(\va),\Dict_p)=\ell_{\tau,q}.
 \end{equation}
We will now prove that the approximation spaces for tensor brushlet systems $\mathcal{W}$ with approximation error measured in a suitable $\alpha$-modulation spaces can be identified with 
smoothness spaces for a suitably normalized version of $\mathcal{W}$. The  smoothness spaces can, in certain favorable cases, be identified with  
$\alpha$-modulation spaces.

\begin{myprop}\label{prop:equiv}
 Let $0\leq \alpha<1$, and let $\{w_{R,\cn}\}_{R\in \cK,\cn\in \N_0^d}$ be the orthonormal brushlet system 
associated with the $\alpha$-covering considered in Proposition \ref{prop:onb}. Let
$\Dict=\big\{{w_{R,\cn}}/{\|\mathbf{1}_{(R,\cn)}\|_{m^{s,\alpha}_{p,p}(\va)}}\big\}_{R\in\cK,
  \cn\in\N_0^d}$ for some $s> 0$ and $0<p<\infty$. Then
 $$\mathcal{A}_{\, q}^\gamma\big(M^{s,\alpha}_{p,p}(\va),\Dict\big) =
   \K^{\tau}_{q}\big(M^{s,\alpha}_{p,p}(\va),\Dict\big),\qquad
   \frac{1}{\tau}=\gamma+\frac{1}{p},\,\gamma>0,\, 0<q\leq \infty,$$
 with equivalent norms. Moreover, for $\tau<p$,
$$
\K^{\tau}_{ \tau}\big(M^{s,\alpha}_{p,p}(\va)),\Dict\big)=M^{\beta,\alpha}_{\tau,\tau}(\va),\qquad
\text{ with }\quad \beta=\frac{\nu}{\tau}-\frac{\nu}{p}+s.$$
\end{myprop}

\begin{proof}
The proof of the first claim relies on the  formalism of DeVore and Popov \cite{MR942269}. In the present setup, we may  in fact call directly on \cite[Theorem 6.1]{Garrigos2004}, using that $\widetilde{\Dict_p}$ satisfies the stability condition \eqref{eq:norm}, to conclude that
$$\mathcal{A}_{\, q}^\gamma\big(m^{s,\alpha}_{p,p}(\va),\widetilde{\Dict_p}\big)=\ell_{\tau,q},\qquad
\frac{1}{\tau}=\gamma+\frac{1}{p},\,\gamma>0,\, 0<q\leq \infty.
$$
Then, according to \eqref{eq:approxs} and \eqref{eq:Kdisk},
$$\mathcal{A}_{\, q}^\gamma\big(M^{s,\alpha}_{p,p}(\va))=R_{\cW}\mathcal{A}_{\, q}^\gamma\big(m^{s,\alpha}_{p,p}(\va))=R_{\cW}\ell_{\tau,q}=\K^{\tau}_{q}\big(M^{s,\alpha}_{p,p}(\va),\Dict\big).$$
To prove the  second
claim, we use the fact that $\|\mathbf{1}_{(R,\cn)}\|_{m^{s,\alpha}_{p,p}(\va)}\asymp |R|^{s/\nu+1/2-1/p}$. For $f\in \K^{\tau}_{ \tau}(M^{s,\alpha}_{p,p}(\va),\Dict)$, we therefore have
\begin{equation}
  \label{eq:normm}
  \|f\|_{\K^{\tau}_{ \tau}(M^{s,\alpha}_{p,p}(\va)),\Dict)}\asymp
\bigg(\sum_{R,\cn}|R|^{(\frac{s}{\nu}+\frac{1}2-\frac{1}p)\tau}|\langle f,w_{R,\cn}\rangle|^\tau\bigg)^{1/\tau}.
\end{equation}
Suppose $\tau<p$ and define $\beta>0$  by the equation $\beta/\nu+1/2-1/\tau=s/\nu+ 1/2-1/p\Rightarrow \beta=\nu/\tau-\nu/p+s$. By
Theorem \ref{prop:modu} and Eq.\ \eqref{eq:normm},
\begin{align*}
\|f\|_{M^{\beta,\alpha}_{\tau,\tau}(\va)}&\asymp\|\{\langle f,w_{R,\cn}\}\|_{{m}^{s,\alpha}_{\tau,\tau}(\vec{a})}\\&\asymp \bigg(\sum_{R\in\mathcal{K}}|R|^{(\frac{\beta}{\nu}+\frac{1}2-\frac{1}\tau)\tau} \sum_{\cn\in\bN_0^d} |\langle
f,w_{R,\cn}\rangle|^\tau\bigg)^{1/\tau}\\
&=\bigg(\sum_{R\in\mathcal{K}}|R|^{(\frac{s}{\nu}+\frac{1}2-\frac{1}p)\tau} \sum_{\cn\in\bN_0^d} |\langle
f,w_{R,\cn}\rangle|^\tau\bigg)^{1/\tau}\\
&\asymp 
\|f\|_{\K^{\tau}_{ \tau}(M^{s,\alpha}_{p,p}(\va)),\Dict)}.
\end{align*}
This concludes the proof.
\end{proof}

\subsection{Direct and inverse estimates for $m$-term approximation with tensor product brushlet systems}
In this section we establish Jackson (direct) and Bernstein (inverse) type inequalities
for functions from the $\alpha$-modulation spaces, with the error
measured in suitable $\alpha$-Triebel-Lizorkin spaces. We
will rely on  the same general type of arguments as used by Kyriazis in
\cite{MR1866250}, where similar estimates were established for non-redundant wavelet systems. However, in all estimates we have to take into account the specific polynomial structure of the $\alpha$-coverings.

\subsection{Jackson inequality}
In order to prove a Jackson inequality, we  first need to prove a technical lemma that will be
of use in obtaining both the Jackson and Bernstein estimate.  Lemma \ref{lem:Imax} is perhaps the most important contribution in this section as it shows how to handle the specific geometry of the $\alpha$-coverings. We mention that the lemma takes into account the (rather significant) point-wise overlap of the sets $\{U(R,\cn): R\in\mathcal{K}_j,\cn\in\bN_0^d\}$. 

\begin{mylemma}\label{lem:Imax}
Given $0<\alpha< 1$, let $\beta:=(1-\alpha)^{-1}$ and let $\cK$ be the disjoint $\alpha$-covering considered in Lemma \ref{le:cut}.
For $\Lambda \subset \cK\times \N_0^d$, with $\# \Lambda <\infty$, we define
 $I_{\Lambda}(\cx) = \max \{|R|\,\mathbf{1}_{U(R,\cn)}(\cx):  (R,\cn)\in \Lambda\}.$ 
Then, for any $q>0$ there exists a constant $C:=C(q)$, independent of $\Lambda$, such that,
  $$\sum_{(R,\cn)\in\Lambda}|R|^{q-\nu\frac{1-\alpha}{\alpha}}\mathbf{1}_{U(R,\cn)}(\cx) \leq C
  I_{\Lambda}(\cx)^{q},$$
for any $q>0$.
\end{mylemma}
\begin{proof}
The proof relies on a counting argument. Let $\Lambda_j=\{(R,\cn)\in\Lambda: R\in \mathcal{K}_j\}$, where 
we recall from Lemma \ref{le:cut} that $R\in \cK_j\Rightarrow |R|\asymp j^{\nu(\beta-1)}$. We also notice that for any $R\in \cK_j$,  $\# \cK_j\asymp |K_j|/|R|\asymp j^{\nu\beta-1}/j^{\nu(\beta-1)}=j^{\nu-1}$. Put $$I_j(\cx)=
\begin{cases}\max\{j: \cx\in \cup_{(R,\cn)\in\Lambda_j}U(R,\cn) \},& \cx\in\bigcup_{(R,\cn)\in\Lambda} U(R,\cn),\\
0,&\text{otherwise.}
\end{cases}
$$ For convenience, we also let $a:=\nu\frac{1-\alpha}{\alpha}=\frac{\nu}{\beta-1}$.  Then, for fixed $\cx\in\bR^d$,
\begin{align*}
\sum_{(R,\cn)\in\Lambda}|R|^{q-a}\mathbf{1}_{U(R,\cn)}(\cx) &=\sum_{j=0}^\infty \sum_{(R,\cn)\in\Lambda_j}|R|^{q-a}\mathbf{1}_{U(R,\cn)}(\cx)\\&
\leq C
\sum_{j=0}^{I_j(\cx)} j^{\nu(\beta-1)\cdot(q-a)}\sum_{(R,\cn)\in\Lambda_j}\mathbf{1}_{U(R,\cn)}(\cx)\\
&\leq CL
\sum_{j=0}^{I_j(\cx)} j^{\nu(\beta-1)\cdot(q-a)}\cdot j^{\nu-1}\\
&\leq C' I_j(\cx)^{\nu(\beta-1)\cdot(q-a)+\nu}\\
&=C' I_j(\cx)^{\nu(\beta-1)\cdot[(q-a)+\frac{\nu}{\beta-1}]},
\end{align*}
where we used the estimate from Eq.\ \ref{eq:L1} and the elementary observation that $\sum_{j=1}^N j^\eta=O(N^{\eta+1})$ for $\eta>0$.
Now, notice that $I_j(\cx)^{\nu(\beta-1)}\asymp I_\Lambda(\cx)$, and that $a=\frac{\nu}{\beta-1}$, so
$$\sum_{(R,\cn)\in\Lambda}|R|^{q-a}\mathbf{1}_{U(R,\cn)}(\cx)\leq C''I_\Lambda(\cx)^q,$$
which completes the proof of the lemma.
\end{proof}

We have the following Jackson estimates for $m$-term brushlet
approximation to functions in the
$\alpha$-modulation space  $M^{\gamma,\alpha}_{\tau,\tau}(\va)$, where the
error is measured in the $\alpha$-Triebel-Lizorkin space $F^{\beta,\alpha}_{p,t}(\va)$.

\begin{myprop}\label{prop:jackson}
Let $0<\alpha < 1$, $0<\tau<p<\infty$, $0<t<\infty$,
  $-\infty<\beta<\gamma<\infty$, and let $\Dict:=\cW_{(1-\alpha)^{-1}}$ be 
the tensor-brushlet system considered in Proposition \ref{prop:onb}. Define $r$ by 
$$r = r(\nu,\alpha,p,t):=\begin{cases} 0& \text{for}\quad t\geq p,\\
  \nu\frac{1-\alpha}{\alpha} & \text{for}\quad t<p.\end{cases}$$
Suppose $1/\tau-1/p=(\gamma-\beta)/\nu -r/t$. Then there exists a constant $C$ such that for every $f\in 
  M^{\gamma,\alpha}_{\tau,\tau}(\va)$ and $m\in\bN$,
\begin{equation}\label{eq:jack}
\sigma_m(f,\Dict)_{F^{\beta,\alpha}_{p,t}(\va)}\leq
Cm^{-\frac{\gamma-\beta}{\nu}+\frac rt}
\|f\|_{M^{\gamma,\alpha}_{\tau,\tau}(\va)}.
\end{equation}
\end{myprop}

\begin{myre}
 We mention that the correction factor $r$ in Proposition \ref{prop:jackson} must be incorporated as a consequence of the polynomial structure of the $\alpha$-coverings, where we notice that $r\rightarrow 0$ as $\alpha\rightarrow 1$. In fact, we have no correction in the limit case $\alpha=1$. The case $\alpha=1$ is not treated in the present paper, but it corresponds exactly to a dyadic wavelet setup considered by the author in \cite{MR4635619}.
\end{myre}

\begin{proof}
Let $f\in M^{\gamma,\alpha}_{\tau,\tau}(\va)$. We put $c_{R,\cn}(f):=
\langle f,w_{R,\cn}\rangle |R|^{\gamma/\nu-1/\tau+1/2}$. Then
$$\|f\|_{M^{\gamma,\alpha}_{\tau,\tau}(\va)} \asymp
\|\{c_{R,\cn}(f)\}\|_{\ell_\tau}:=M.$$ Notice also that $|\langle
f,w_{R,\cn}\rangle||R|^{\beta/\nu+1/2} = |c_{R,\cn}(f)||R|^{1/p-r/t}$.
Define for $j\in \bZ$
$$\Lambda_j=\big\{(R,\cn)\in \cK\times\N_0^d\colon 2^{-j}<|c_{R,\cn}(f)|\leq 2^{-j+1}\big\}.$$
Standard estimates show that
$$\#\Lambda_j\leq CM^\tau 2^{j\tau}\quad \Rightarrow\quad \sum_{j\leq
  k}\#\Lambda_j\leq C M^\tau2^{k\tau}.$$
Put $$T_k=\sum_{j\leq k}\sum_{(R,\cn)\in\Lambda_j} \langle
f,w_{R,\cn}\rangle w_{R,\cn},$$
then $T_k\in\Sigma_{\lceil C M^\tau2^{k\tau} \rceil}(\Dict)$.
 Since
$\sigma_m(f,\Dict)_{F^{\beta,\alpha}_{p,t}(\va)}$ is decreasing as a function of $m$, it
suffices to prove (\ref{eq:jack}) for the  subsequence $m_k=\lceil C M^\tau2^{k\tau} \rceil$ as $m_{k+1}/m_k\asymp 1$ for $k\rightarrow \infty$.
Specifically, we will prove that
$$\|f-T_k\|_{F^{\beta,\alpha}_{p,t}(\va)}\leq
C(M^\tau
2^{k\tau})^{-(\frac 1\tau -\frac 1p)}\|f\|_{M^{\gamma,\alpha}_{\tau,\tau}(\va)}
=C(M^{\tau}2^{-k(p-\tau)})^{\frac 1p}.$$
By the norm equivalence of Theorem \ref{prop:modu}, we obtain
\begin{align*}
 \|f-T_k\|_{F^{\beta,\alpha}_{p,t}(\va)}^p &= \int_{\R^d}
\bigg(\sum_{j\geq k+1} \sum_{(R,\cn)\in \Lambda_j} \big(|\langle
f,w_{R,\cn}\rangle||R|^{\beta/\nu+1/2}
\mathbf{1}_{U_{(R,\cn)}}\big)^t\bigg)^{\frac{p}{t}} \dx\\
&= \int_{\R^d} \bigg(\sum_{j\geq k+1} \sum_{(R,\cn)\in \Lambda_j}
\big(|c_{R,\cn}(f)||R|^{1/p-r/t} \mathbf{1}_{U(R,\cn)}\big)^t\bigg)^{\frac{p}{t}}
\dx\\
&\leq C\int_{\R^d} \bigg(\sum_{j\geq k+1} \sum_{(R,\cn)\in \Lambda_j}
\big(2^{-j}|R|^{1/p-r/t} \mathbf{1}_{U(R,\cn)}\big)^t\bigg)^{\frac{p}{t}}
\dx.
\end{align*}
Two distinct cases need to be considered. The first case is $p\leq
t$, where the Jackson inequality is universally true. We have,
\begin{align*}
\|f-T_k\|_{F^{\beta,\alpha}_{p,t}(\va)}^p
&\leq C\sum_{j\geq k+1} \sum_{(n,I)\in \Lambda_j} \int_{\bR^d}
\big(2^{-j}|R|^{1/p-r/t} \mathbf{1}_{U(R,\cn)}\big)^p \dx\\
&= C\sum_{j\geq k+1} 2^{-jp}\int_{\bR^d}\sum_{(R,\cn)\in
\Lambda_j} |R| \,\mathbf{1}_{U(R,\cn)} \dx\\
&\leq C'\sum_{j\geq k+1} 2^{-jp}\# \Lambda_j, \qquad \text{since}\: |U(R,\cn)|\asymp |R|^{-1},\\
&\leq C M^\tau 2^{-k(p-\tau)}.
\end{align*}
Next, we consider $p>t$. We have
\begin{align*}
  \|f-T_k\|_{F^{\beta,\alpha}_{p,t}(\va)}^t
  &\leq C\Big\|\sum_{j\geq k+1} \sum_{(R,\cn)\in
    \Lambda_j} 2^{-jt}|R|^{\frac tp-r}
  \mathbf{1}_{U(R,\cn)}\Big\|_{L_{\frac{p}{t}}}\\
  &\leq C\sum_{j\geq k+1} 2^{-jt} \Big\|\sum_{(R,\cn)\in
    \Lambda_j} |R|^{\frac tp-r}
  \mathbf{1}_{U(R,\cn)}\Big\|_{L_{\frac{p}{t}}}.
\end{align*}
Lemma \ref{lem:Imax} yields
\begin{align*}
      \Big\|\sum_{(R,\cn)\in \Lambda_j} |R|^{\frac tp-r}
  \mathbf{1}_{U(R,\cn)}\Big\|_{L_{\frac{p}{t}}}
  &\leq C\big\| I_{\Lambda_j}(\cdot)^{\frac tp}
  \big\|_{L_{\frac{p}{t}}}\\
&  = C\biggl( \int_{\R^d} I_{\Lambda_j}(\cx)\dx\biggr)^{\frac tp}\\
&\leq \biggl( \int_{\R^d} \sum_{(R,\cn)\in\Lambda_j} |R|\,\mathbf{1}_{U(R,\cn)}(\cx)\dx\biggr)^{\frac tp}\\
&  \leq C(\# \Lambda_j)^{\frac tp},
\end{align*}
and thus
\begin{align*}
  \|f-T_k\|_{F^{\beta,\alpha}_{p,t}(\va)}^t
  &\leq C\sum_{j\geq k+1} 2^{-jt}(\# \Lambda_j)^{\frac tp} \leq
  C'M^{\tau \frac tp} \sum_{j\geq k+1} 2^{-j\frac tp(p-\tau)} \leq C''
  \big( M^\tau 2^{-k(p-\tau)}\big)^{\frac tp}.
\end{align*}
\end{proof}

\begin{myre}
The fact that $\Dict$ is non-redundant is not used at all in the proof of Proposition \ref{prop:jackson}. In fact, Proposition \ref{prop:jackson} will hold for any redundant frame of the same structure as $\Dict$, provided the frame satisfies similar norm characterizations as given in Theorem \ref{prop:modu}, but based on the canonical frame coefficients.
\end{myre}
\subsection{Bernstein inequalities}
We can also establish a Bernstein-type inequality for the
$\alpha$-Trie\-bel-Li\-zor\-kin and $\alpha$-modulation spaces. The estimate relies heavily on the \textit{non-redundancy} of the constructed tensor brushlet bases.
The first result concerns $m$-term brushlet approximation to functions
in the $\alpha$-modulation space  $M^{\gamma,\alpha}_{\tau,\tau}(\va)$,
where the error is measured in $M^{\gamma,\alpha}_{p,t}(\va)$.

\begin{myprop}\label{prop:bernstein1}
Given $0\leq \alpha < 1$, and $0<\tau<\infty$, let $\Dict:=\cW_{(1-\alpha)^{-1}}$ be 
the tensor-brushlet system considered in Proposition \ref{prop:onb}.
Let $\tau<p<\infty$, $0<q<t<\infty$, and $-\infty<\beta<\gamma<\infty$. Suppose
\begin{equation}\label{eq:pqr}
\frac 1\tau -\frac 1p = \frac 1q -\frac 1t = \frac{\gamma
  -\beta}\nu,
\end{equation}
then for every $g\in \Sigma_n(\Dict)$
$$\|g\|_{M^{\gamma,\alpha}_{\tau,q}(\va)} \leq Cn^{ \frac{\gamma
    -\beta}\nu} \|g\|_{M^{\gamma,\alpha}_{\beta,t}(\va)}.$$
\end{myprop}

\begin{proof}
Let $g= \sum_{(R,\cn)\in \Lambda} c_{R,\cn}w_{R,\cn}$, with $\#
\Lambda =n$, $w_{R,\cn}\in \Dict$, and define for every $R\in \cK$,
$\Lambda_R := \{ \cn\in
\N_0^d\colon (R,\cn)\in \Lambda\}$. Then since $p>\tau$ and $t>q$ we can use
H\"olders inequality together with the relation \eqref{eq:pqr} to obtain,

\begin{align*}
  \|g\|_{M^{\gamma,\alpha}_{\tau,q}(\va)}^q &\leq C\sum_{R\in \cK}
  \biggl(\sum_{\cn\in \Lambda_R} (|R|^{\frac{\gamma}{\nu} +\frac 12
  -\frac 1{\tau}} |c_{R,\cn}|)^\tau \biggr)^{\frac q\tau}\\
&= C\sum_{R\in \cK}
  \biggl(\sum_{\cn\in \Lambda_R} (|R|^{\frac{\beta}{\nu} +\frac 12
  -\frac 1p} |c_{R,\cn}|)^\tau \biggr)^{\frac q\tau}\\
&\leq C\sum_{R\in \cK} (\# \Lambda_R)^{q(\frac 1\tau-\frac 1p)} \biggl(
  \sum_{\cn\in \Lambda_R} (|R|^{\frac{\beta}{\nu} +\frac 12 -\frac
  1p} |c_{R,\cn}|)^p \biggr)^{\frac qp}\\
&\leq C\biggl( \sum_{R\in \cK} \big(\# \Lambda_R\big)^{q\big(\frac1\tau-\frac1{p}\big)\big(q\big(\frac1{q}-\frac1{t}\big)\big)^{-1}} \biggr)^{1-
  \frac{q}{t}}
  \biggl( \sum_{R\in \cK} \biggl( \sum_{\cn\in \Lambda_R}
  (|R|^{\frac{\beta}{\nu} +\frac 12 -\frac 1p} |c_{R,\cn}|)^p
  \biggr)^{\frac tp} \biggr)^{\frac qt}\\
&= (\# \Lambda)^{q(\frac{1}{q} -\frac{1}{t})} \|g\|_{M^{\beta,\alpha}_{p,t}(\va)}^q \\
&=
  n^{q\frac{\gamma -\beta}\nu} \|g\|_{M^{\beta,\alpha}_{p,t}(\va)}^q.
\end{align*}
\end{proof}

Recall that Eq.\ \eqref{eq:emb} provides  the general embedding $F^{\beta,\alpha}_{p,t}(\va) \hookrightarrow
M^{\beta,\alpha}_{p,\max\{p,t\}}(\va)$, which can be used directly to derive the following result.
\begin{mycor}\label{cor:bernstein}
  Let $0\leq \alpha < 1$,
  $0<p\leq t<\infty$, and $\beta<\gamma$.  Let $\Dict:=\cW_{(1-\alpha)^{-1}}$ be 
the tensor-brushlet system considered in Proposition \ref{prop:onb}. Suppose
$$\frac 1\tau -\frac 1p = \frac 1q -\frac 1t = \frac{\gamma
  -\beta}\nu.$$
Then for every $g\in \Sigma_n(\Dict)$,
$$\|g\|_{M^{\gamma,\alpha}_{\tau,q}(\va)} \leq Cn^{\TS \frac{\gamma
    -\beta}\nu} \|g\|_{F^{\beta,\alpha}_{p,t}(\va)}.$$
\end{mycor}

For  general $p,t\in (0,\infty)$ we cannot hope for as good a
Bernstein inequality as in the previous corollary. However, we have
\begin{myprop}\label{prop:bernstein2}
Let $0<\alpha < 1$ and let $\Dict:=\cW_{(1-\alpha)^{-1}}$ be 
the tensor-brushlet system considered in Proposition \ref{prop:onb}.
Suppose $0<p<\infty$, $0< t<\infty$, and
$\beta<\gamma$.  Define $\tau$ by ${\TS \frac 1{\tau}} =
{\TS \frac{\gamma -\beta}\nu} +{\TS \frac 1p}$. Then,
  for every $g\in \Sigma_n(\Dict)$,
$$\|g\|_{M^{\gamma-\nu^2\frac{1-\alpha}{\tau\alpha},\alpha}_{\tau,\tau}(\va)} \leq C
n^{\frac{\gamma-\beta}{\nu}} \|g\|_{F^{\beta,\alpha}_{p,t}(\va)}.$$
\end{myprop}

\begin{proof}
Suppose $g= \sum_{(R,\cn)\in \Lambda} c_{R,\cn}w_{R,\cn}$, with $\#
\Lambda =n$, $w_{R,\cn}\in \Dict$, and let $\mathcal{S}_q^s(g)$ be defined as in Eq.\ \eqref{eq:Sq}.
Then, since $\tau<p$ we have
\begin{align*}
  \|g\|_{M^{\gamma-\nu^2\frac{1-\alpha}{\tau\alpha},\alpha}_{\tau,\tau}(\va)}^\tau &\leq C\sum_{(R,\cn)\in
  \Lambda} (|R|^{\frac{\gamma}{\nu}-\nu\frac{1-\alpha}{\tau\alpha}+1/2-1/\tau} |c_{R,\cn}|)^\tau \\
&= C\int_{\R^d} \sum_{(R,\cn)\in
  \Lambda} (|R|^{\frac
  \beta\nu +\frac12 } |c_{R,\cn}|)^\tau \cdot |R|^{\frac{\tau(\gamma-\beta)}{\nu}-\nu\frac{1-\alpha}\alpha} \mathbf{1}_{U(R,\cn)}(\cx)\dx \\
&\leq C\int_{\R^d} \Bigl(  \mathcal{S}_t^\beta(g)(\cx))^t \Bigr)^{\frac{\tau}t} \cdot \Bigl( \sum_{(R,\cn)\in
  \Lambda} |R|^{\frac{\tau(\gamma-\beta)}{\nu}-\nu\frac{1-\alpha}\alpha}
  \mathbf{1}_{U(R,\cn)}(\cx)\Bigr)\dx\\
&\leq 2C\biggl\| \Bigl((\mathcal{S}_t^\beta(g))^t\Bigr)^{\frac{\tau}t} \biggr\|_{L_{\frac p\tau}}\cdot \biggl\|
  \sum_{(R,\cn)\in \Lambda} |R|^{\frac{\tau(\gamma-\beta)}{\nu}-\nu\frac{1-\alpha}\alpha}
  \mathbf{1}_{U(R,\cn)}(\cdot)\biggr\|_{L_{\frac p{p-\tau}}}.
\end{align*}
By Lemma \ref{lem:Imax} we have 
\begin{align*}
  \int_{\R^d}
  \Bigl( \sum_{(R,\cn)\in \Lambda} |R|^{\frac{\tau(\gamma-\beta)}{\nu}-\nu\frac{1-\alpha}\alpha}
  \mathbf{1}_{U(R,\cn)}(\cx)\Bigr)^{\frac p{p-\tau}} \dx 
&\leq C\int_{\R^d} I_\Lambda(\cx)^{\frac{\tau(\gamma-\beta)}{\nu}\cdot
  \frac{p}{p-\tau}}\dx \\
&= C \int_{\R^d} I_\Lambda(\cx)\dx \\&\leq C\# \Lambda .
\end{align*}
Using Proposition \ref{prop:normchar}, we conclude that
$$\|g\|_{M^{\gamma -\nu^2\frac{1-\alpha}{\tau\alpha},\alpha}_{\tau,\tau}(\va)}^\tau \leq C' (\#
\Lambda)^{\frac{p-\tau}{p}} \bigl\| \mathcal{S}_t^\beta(g)^{\tau} \bigr\|_{L_{\frac p\tau}}
= C'n^{\tau\frac{\gamma -\beta}{\nu}} \|g\|_{F^{\beta,\alpha}_{p,t}(\va)}^\tau .$$
\end{proof}

\subsection{Additional Jackson and Bernstein estimates}

The approximation framework developed by DeVore and Popov \cite{MR942269} 
makes it clear that the main tool in the characterization of
$\mathcal{A}^\gamma_{\, q}(X,\Dict)$ comes from the link between
approximation theory and interpolation theory (see also
\cite[Theorem 7.9.1]{DeVore1993}). Let $Y$ be a (quasi-)Banach
space with semi-(quasi)norm $|\cdot|_{Y}$ continuously embedded in
in the (quasi-)Banach space $X$. Given $r_J>0$,
a Jackson inequality
\begin{equation}
   \label{eq:JacksonIneqGeneral}
   \sigma_m(f,\Dict)_X
    \leq
   C m^{-r_J} |f|_{Y},
   \qquad\forall f \in Y: \forall m\in \N,
\end{equation}
with constant $C$  independent of $f$, $S$ and $m$,
 implies 
the continuous embedding
\begin{equation}
   \label{eq:JacksonEmb}
\left(X,Y\right)_{\sigma/r_J,q}
 \hookrightarrow \mathcal{A}^{\sigma}_{\, q}(X,\Dict),
 \end{equation}
for all $0 <
\sigma < r_J$, $q \in (0,\infty]$, while 
a Bernstein inequality for $r_B>0$
\begin{equation}
  \label{eq:BernsteinIneqGeneral}
   |S|_{Y}
   \leq C' m^{r_B} \|S\|_X,
  \qquad\qquad\forall S \in \Sigma_m(\Dict),
\end{equation}
with constant  $C'$ independent of $f$, $S$ and $m$,
 implies the continuous embedding
\begin{equation}
  \label{eq:BernsteinEmb}
 \mathcal{A}^{\sigma}_{\, q}(X,\Dict)
 \hookrightarrow \left(X,Y\right)_{\sigma/r_B,q}
 \end{equation}
for all $0 <
\sigma < r_B$ and $q \in (0,\infty]$. As is well-known, this leads to a full characterization of 
$\mathcal{A}^{s}_{\, q}(X,\Dict)$ an an interpolation space in the case where the Jackson and Bernstein inequality have matching exponents, i.e., $r_J=r_B$. We refer to  \cite[§7]{DeVore1993} for a more detailed discussion on these general embedding results.

Now, using the Jackson and Bernstein inequalities from
Propositions \ref{prop:jackson} and \ref{prop:bernstein2}, and from
Corollary \ref{cor:bernstein}, we get the following embeddings, which also concludes the paper.

\begin{myprop}\label{prop:int}
  Let $0<\alpha <1$, $0<p<\infty$, $0\leq t<\infty$, and
  $\beta<\gamma$. Let $\Dict:=\cW_{(1-\alpha)^{-1}}$ be 
the tensor-brushlet system considered in Proposition \ref{prop:onb}.Define 
  $\tau$ by $1/\tau -1/p =1/\eta -1/t = (\gamma -\beta)/\nu$.
  For $t\geq p$, we have the Jackson embedding
  \begin{equation}
  \label{eq:aJEmb}
  \left(F^{\beta,\alpha}_{p,t}(\va),M^{\gamma,\alpha}_{\tau,\tau}(\va)\right)_{\frac{ s}{\gamma-\beta},q}
  \hookrightarrow
  \mathcal{A}^{s/\nu}_{\, q}\big(F^{\beta,\alpha}_{p,t}(\va),\Dict\big),\qquad
  s<{\gamma-\beta},
   \end{equation}
  and the Bernstein embedding
    \begin{equation}
  \label{eq:aBEmb}
  \mathcal{A}^{s/\nu}_{\, q}\big(F^{\beta,\alpha}_{p,t}(\va),\Dict\big)
  \hookrightarrow 
  \left(F^{\beta,\alpha}_{p,t}(\va),M^{\gamma,\alpha}_{\tau,\eta}(\va)\right)_{\frac{ s}{\gamma-\beta},q},\qquad
  s<{\gamma-\beta}.
  \end{equation}
  For $t<p$, we have the weaker Jackson embedding:
  \begin{equation}\label{eq:weakjackson}
    \left(F^{\beta,\alpha}_{p,t}(\va),M^{\gamma,\alpha}_{\tau',\tau'}(\va)\right)_{\frac{s}{v\nu},q}
    \hookrightarrow
    \mathcal{A}^{s/\nu}_{\, q}\big(F^{\beta,\alpha}_{p,t}(\va),\Dict\big), 
  \end{equation}
  for $s<v\nu$, where ${\TS \frac{1}{\tau'} = \frac 1\tau
    -\nu\frac{1-\alpha}{t\alpha}}$, provided $v = \frac 1{\tau'} -\frac 1p>0$.
Finally, for $t<p$, we have the Bernstein embedding
  \begin{equation}\label{eq:weakbernstein}
    \mathcal{A}^{s/\nu}_{\, q}\big(F^{\beta,\alpha}_{p,t}(\va),\Dict\big) \hookrightarrow
    \left(F^{\beta,\alpha}_{p,t}(\va),M^{\gamma',\alpha}_{\tau,\tau}(\va)\right)_{\frac{ s}{\gamma-\beta},q}
  \end{equation}
  for $s<\gamma-\beta$, with
  $\gamma' = \gamma -\nu^2(1-\alpha)/{\tau\alpha}$.
\end{myprop}

\begin{proof}
 We first consider the case where $t\geq p$.  Define 
  $\tau$ by $1/\tau -1/p =1/\eta -1/t = (\gamma -\beta)/\nu$, and
 let $X=F^{\beta,\alpha}_{p,t}(\va)$ and $Y=M^{\gamma,\alpha}_{\tau,\tau}(\va)$. From Proposition \ref{prop:jackson}, we have the Jackson estimate \eqref{eq:JacksonIneqGeneral} 
 with $r_J=\frac{\gamma-\beta}{\nu}$. Put $s=\sigma\nu$ and \eqref{eq:aJEmb} follows directly from \eqref{eq:JacksonEmb}. Next, we put  $Y=M^{\gamma,\alpha}_{\tau,\eta}(\va)$, and notice that Corollary \ref{cor:bernstein} provides a Bernstein inequality of the type \eqref{eq:BernsteinIneqGeneral} with $r_B=\frac{\gamma-\beta}{\nu}$. Hence, we may put $s=\sigma\nu$ and \eqref{eq:aBEmb} follows from \eqref{eq:BernsteinEmb}.

 Next, we consider the remaining case $t<p$. Put $X=F^{\beta,\alpha}_{p,t}(\va)$ and $Y=M^{\gamma,\alpha}_{\tau',\tau'}(\va)$, where we let ${\TS \frac{1}{\tau'} = \frac 1\tau
    -\nu\frac{1-\alpha}{t\alpha}}$ and $v = \frac 1{\tau'} -\frac 1p$. From Proposition \ref{prop:jackson}, we have the Jackson estimate \eqref{eq:JacksonIneqGeneral} 
 with $r_J=\frac{\gamma-\beta}{\nu}-\nu\frac{1-\alpha}{t\alpha}=v$. We now put $\sigma=s\nu$ and \eqref{eq:weakjackson} follows from \eqref{eq:BernsteinEmb} provided $v>0$. Finally, we let $\gamma' = \gamma -\nu^2\frac{1-\alpha}{\tau\alpha}$ and put 
 $Y=M^{\gamma',\alpha}_{\tau,\tau}(\va)$. From Proposition \ref{prop:bernstein2} we obtain the Bernstein estimate \eqref{eq:BernsteinIneqGeneral} with $r_B=\frac{\gamma-\beta}{\nu}$. Hence, \eqref{eq:weakbernstein} follows directly from the general embedding in \eqref{eq:BernsteinEmb} with $s=\sigma\nu$.
\end{proof}

\appendix
\section{Univariate brushlets} \label{app:brush}\label{s:app}
For the benefit of the reader, we will review various known properties of brushlet  and tensor product brushlet systems in this appendix. 
A univariate brushlet basis is associated with a partition of the
frequency axis $\bR$. The partition can  be chosen with almost
no restrictions, but in order to have good properties of the
associated basis we need to impose some growth conditions on the
partition. 
We have the following definition.

\begin{mydef}
A family $\ptt$ of intervals is called a {\it disjoint covering} of $\R$ if
it consists of a countable set of pairwise disjoint half-open intervals
$I=[\alpha_I,\alpha'_I)$, $\alpha_I<\alpha'_I$, such
that $\cup_{I\in \ptt}I=\R$. If, furthermore, each interval in $\ptt$
has a unique adjacent interval in $\ptt$ to the left and to the right,
and there exists a constant $A>1$ such that
\begin{equation}\label{eq:growthcond}
A^{-1}\leq \frac{|I|}{|I'|} \leq A,\qquad \text{for
  all adjacent}\; I,I'\in \ptt,
\end{equation}
we call $\ptt$ a {\it moderate disjoint covering} of $\R$.
\end{mydef}

Given a moderate disjoint covering $\ptt$ of $\R$, assign to each interval $I\in \ptt$ a cutoff
radius $\varepsilon_I>0$ at the left endpoint and a cutoff radius
$\varepsilon'_I>0$ at the right endpoint, satisfying
\begin{equation}\label{eq:epsilon}
\begin{cases}
\text{(i)}& \ve'_I= \ve_{I'}\; \text{whenever}\;
\alpha'_I=\alpha_{I'}\\
\text{(ii)}& \ve_I+\ve'_I\leq |I|\\
\text{(iii)}& \ve_I\geq c |I|,
\end{cases}
\end{equation}
with $c>0$ independent of $I$.

\begin{myre}
As will be clear in the definition of the brushlet system below, it is not essential that 
exactly half-open intervals are used in the definition of $\ptt$ as the only  parameters used from $I=[\alpha_I,\alpha'_I)\in\ptt$ are the knots  $\alpha_I$ and $\alpha'_I$. Hence, a mix of half-open  and closed intervals may also be used if we drop the requirement of a perfect partition of $\bR$ and allow suitable single point intersections between  pairs of sets from $\ptt$.
\end{myre}

We are now ready to define the brushlet system. For each $I\in \ptt$,
we 
will  construct a smooth bell function localized in a
neighborhood of this interval. Take a
non-negative ramp function $\rho \in C^\infty(\R)$ satisfying
\begin{equation}\label{eq:ramp}
\rho(\xi)=\left\{ \begin{array}{ll}
0&\mbox{for}\; \xi \leq -1,\\
1&\mbox{for}\; \xi \geq 1,
\end{array} \right.
\end{equation}
with the property that
\begin{equation}\label{eq:ramp2}
\rho (\xi)^2+\rho(-\xi)^2=1\qquad \text{for all}\; \xi \in\R.
\end{equation}
Define for each $I = [\alpha_I,\alpha'_I)\in \ptt$ the {\it bell function}
\begin{equation}\label{eq:bell}
b_I(\xi) := \rho \bigg( \frac{\xi-\alpha_I}{\ve_I}\bigg)\rho
\bigg( \frac{\alpha'_I-\xi}{\ve'_I}\bigg) .
\end{equation}
Notice that $\supp (b_I) \subset [\alpha_I-\ve_I,\alpha'_I+\ve'_I]$ and
$b_I(\xi)=1$ for $\xi \in [\alpha_I+\ve_I,\alpha'_I-\ve'_I]$.
Now the set of local cosine functions
\begin{equation}\label{eq:brush1}
\widehat{w_{I,n}}(\xi) =\sqrt{\frac{2}{|I|}} b_I(\xi)
\cos\biggl( \pi \bigl( n+{\textstyle \frac12}\bigr)
\frac{\xi- \alpha_I}{|I|} \biggr) ,\quad n\in \N_0,\quad I\in \ptt,
\end{equation}
 with $\bN_0:=\bN\cup \{0\}$, constitute an orthonormal basis for $L_2(\bR)$, see, e.g.,
\cite{Auscher1992}. We 
call the collection $\{ w_{I,n}\colon I\in\ptt, n\in\N_0\}$ a (univariate)
{\it brushlet system}. The brushlets also have an explicit
representation in the time domain. Define the set of {\it central
  bell functions} $\{ g_I\}_{I\in \ptt}$ by
\begin{equation}\label{eq:gb}
\widehat{g_I}(\xi) := \rho \biggl(
\frac{|I|}{\ve_I} \xi\biggr) \rho \biggl( \frac{|I|}{\ve'_I} (1-\xi)\biggr),
\end{equation}
where one verifies that 
\begin{equation}\label{eq:bi}
b_I(\xi) = \hat{g}_I\bigl(|I|^{-1}(\xi -\alpha_I)\bigr). 
\end{equation}
Let for notational convenience
$$e_{I,n}:= \frac{\pi \bigl( n+{\textstyle
    \frac12}\bigr)}{|I|},\qquad
I\in \ptt,\; n\in \N_0.$$
Then,
\begin{equation}
w_{I,n}(x)
= \sqrt{\frac{|I|}{2}} e^{i\alpha_Ix} \bigl\{ g_I\bigl(
  |I|(x +e_{n, I}) \bigr) + g_I\bigl( |I|(x -e_{n,I} ) \bigr)
  \bigr\}.\label{eq:gw}
\end{equation}

By a straightforward calculation, it can be verified (see
\cite{Borup2003}) that for $r\geq 1$ there exists a constant $C:=C(r)<\infty$, independent
of $I\in \ptt$, such that  
\begin{equation}\label{eq:gdecay}
|g_I(x)|\leq C(1+|x|)^{-r}.
\end{equation}  Thus a
brushlet $w_{n,I}$ essentially consists of two well-localized ``humps'' at
the points $\pm e_{n,I}$.

Given a bell function $b_I$, define an
operator $\mathcal{P}_I:L_2(\bR) \rightarrow L_2(\bR)$ by
\begin{equation}
  \label{eq:proje}
  \widehat{\mathcal{P}_If}(\xi) := b_I(\xi)\bigl[ b_I(\xi)\hat{f}(\xi) +
b_I(2\alpha_I-\xi)\hat{f}(2\alpha_I-\xi)- b_I(2\alpha'_I-\xi)\hat{f}(2\alpha'_I-\xi)\bigr].
\end{equation}
It can be verified that $\mathcal{P}_I$ is an orthogonal projection, mapping
$L_2(\bR)$ onto $$\overline{\Span}\{w_{n,I}\colon n\in \N_0\}.$$  In
Section \ref{s:onb}, some of the finer properties of the
operator given by \eqref{eq:proje} are needed. Let us list properties here, and
refer the reader to \cite[Chap.\ 1]{Hernandez1996} for a more detailed
discussion of the properties of local trigonometric bases.

 Suppose $I = [\alpha_I,\alpha'_I)$
and $J=[\alpha_J,\alpha_J')$ are two adjacent compatible intervals
(i.e., $\alpha'_I=\alpha_J$ and $\epsilon_I'=\epsilon_J$). Then it holds true that
\begin{equation}
  \label{eq:summm}
  \widehat{\mathcal{P}_If}(\xi)+\widehat{\mathcal{P}_Jf}(\xi)=\hat{f}(\xi),\qquad
\xi\in[\alpha_I+\epsilon_I,\alpha_J'-\epsilon_J'],\qquad f\in L_2(\bR).
\end{equation}
We can verify \eqref{eq:summm} using the fact that
$b_I\equiv 1$ on $[\alpha_I+\epsilon_I,\alpha'_I-\epsilon_I']$ and that
$b_J\equiv 1$ on $[\alpha_J+\epsilon_J,\alpha_J'-\epsilon_J']$,
together with the fact that 
$$\text{supp}\big(b_I(\cdot)b_I(2\alpha_I-\cdot)\big)\subseteq
[\alpha_I-\epsilon_I,\alpha_I+\epsilon_I]$$
and
$$\text{supp}\big(b_I(\cdot)b_I(2\alpha_I'-\cdot)\big)\subseteq
[\alpha_I'-\epsilon_I',\alpha_I'+\epsilon_I'].
$$
 For 
$\xi\in[\alpha_I'-\epsilon_I',\alpha_J+\epsilon_J]$ we notice that 
\begin{align}
\widehat{\mathcal{P}_If}(\xi)+&\widehat{\mathcal{P}_Jf}(\xi)=
[b_I^2(\xi)+b_J^2(\xi)(\xi)]\hat{f}(\xi)\notag \\&+
\label{eq:add} b_J(\xi)b_J(2\alpha'-\xi)\hat{f}(2\alpha'-\xi)-
b_I(\xi)b_I(2\alpha'-\xi)\hat{f}(2\alpha'-\xi).  
\end{align}
We can then conclude that \eqref{eq:summm} holds true  using the
following facts (see \cite[Chap.\ 1]{Hernandez1996})
$$b_I(\xi)=b_J(2\alpha_I'-\xi),\qquad b_J(\xi)=b_I(2\alpha_J'-\xi),\qquad
\text{for }\xi\in [\alpha_I'-\epsilon_I',\alpha_J+\epsilon_J],$$
and
$$b_I^2(\xi)+b_J^2(\xi)=1,\qquad \text{for } \xi\in [\alpha_I+\epsilon_I,\alpha_J'-\epsilon_J'].$$
Moreover, 
\begin{equation}\label{eq:proj}\mathcal{P}_I+\mathcal{P}_J=\mathcal{P}_{I\cup J},\qquad
\mathcal{P}_I\mathcal{P}_J=\mathcal{P}_J\mathcal{P}_I=0,\end{equation} with the $\epsilon$-values 
$\epsilon_I$ and $\epsilon_J'$ for $I\cup J$.

\begin{myre}\label{re:proj}
In the special case, $I=[-a,a)$ and $J=[-b,b)$, with $0<a<b$ and cutoff values 
$\epsilon_I=\epsilon'_I<(b-a)/2$ and $\epsilon_J=\epsilon'_J<(b-a)/2$, we may appy \eqref{eq:proj} twice to obtain
$$J=[-b,-a)\cup [-a,a)\cup [a,b)\Longrightarrow \mathcal{P}_J=
 \mathcal{P}_{[-b,-a)}+\mathcal{P}_I+ \mathcal{P}_{[a,b)}.
$$
We call on \eqref{eq:proj} again to conclude that
\begin{equation}\label{eq:nested1}
\mathcal{P}_I\mathcal{P}_J=\mathcal{P}_I(\mathcal{P}_{[-b,-a)}+\mathcal{P}_I+ \mathcal{P}_{[a,b)})=0+\mathcal{P}_I^2+0=\mathcal{P}_I,
\end{equation}
and 
\begin{equation}\label{eq:nested2}
\mathcal{P}_J\mathcal{P}_I=(\mathcal{P}_{[-b,-a)}+\mathcal{P}_I+ \mathcal{P}_{[a,b)})\mathcal{P}_I=0+\mathcal{P}_I^2+0=\mathcal{P}_I.
\end{equation}
\end{myre}
A convenient way to construct multivariate brushlet systems on $\bR^d$, $d\geq 2$, is to use a tensor product approach. For a rectangle $R=I_1\times I_2\times\cdots \times I_d\subset \bR^d$, $d\geq 2$, with 
$I_i = [\alpha_{I_i},\alpha'_{I_i})$, $i=1,\ldots,d$. We define
 $$P_R=\bigotimes_{j=1}^d \mathcal{P}_{I_i}.$$ Clearly, $P_R$
 is a projection operator $$P_R\colon L_2(\bR^d) \rightarrow
 \overline{\operatorname{span}}\bigg\{\bigotimes_{j=1}^d w_{i_j,I_j}\colon i_j\in \bN_0\bigg\}.$$
Notice that, 
\begin{equation}
  \label{eq:operatorP}
  P_R = b_R(D)\Bigl[\bigotimes_{j=1}^d
(\operatorname{Id}+R_{\alpha_{I_j}}-R_{\alpha_{I_j}'})
\Bigr]b_R(D),
\end{equation}
where 
\begin{equation}\label{eq:S}
\widehat{b_R(D)f} := b_R\hat{f},
\end{equation}
with  $b_R:=\bigotimes_{j=1}^d b_{I_j}$, and $R_af(x) := e^{i2a}f(-x)$,  
$c,a\in \bR$. The corresponding orthonormal tensor product basis of brushlets is given by
$$w_{R,\cn}:=\bigotimes_{j=1}^d w_{n_j,I_j},\qquad \cn=(n_1,\ldots, n_d)^T\in \bN_0^d.$$
Let $\delta_R=\text{diag}(|I_i|,|I_2|,\ldots,|I_d|)$. Repeated use of \eqref{eq:gw} yields
\begin{align}
|w_{R,\cn}(\cx)|&\leq  2^{-d/2}|R|^{1/2}\prod_{i=1}^d  \left(\big|g_{I_i}\bigl(
  |I_i|(x_i +e_{I_i,n_i}) \bigr)\big|+ \big|g_{I_i}\bigl( |I_i|(x_i -e_{I_{i},n_i})\bigr)\big| 
  \right)\nonumber\\
  &\leq 2^{-d/2}|R|^{1/2}\sum_{m=1}^{2^d} |G_R(\delta_R(\cx-O_m \ce_{R,\cn}))|,\label{eq:G}
  \end{align}
where 
\begin{equation}\label{eq:Gf}
G_R:=g_{I_1}\otimes\cdots\otimes g_{I_d}
\end{equation}
with $\ce_{R,\cn}:=[e_{I_1,n_1},\ldots,e_{I_d,n_d}]^T$, and $O_m:=\text{diag}(\cv_m)$, with $\cv_m\in\R^d$ chosen such that $\cup_{m=1}^{2^d}\cv_m= \{-1,1\}^d$. We also notice, using \eqref{eq:gdecay}, that for any $r>0$,
$$|G_R(\cx)|\leq C_r(1+|\cx|_{\va})^{-r},\qquad \cx\in\R^d,$$
with $C_r$ a finite constant independent of $R$. Finally, we notice by \eqref{eq:bi} that
\begin{equation}\label{eq:GG}
    b_R(\cdot)=\widehat{G_R}\big(\delta_R^{-1}(\cdot-\bm{\alpha}_R)\big),
\end{equation}
with $\bm{\alpha}_R:=(\alpha_{I_1},\ldots,\alpha_{I_d})^\top$.

\section{Some Estimates for Maximal Functions}

For $f\in L_{1,\text{loc}}(\bR^d)$, the (uncentered) maximal function $\cM$ is defined by
\begin{equation}\label{MK}
\cM f(\cx)=\sup_{B\in\mathcal{B}, \cx\in B} \dfrac{1}{|B|} \int_B |f(\cy)| \dy,\;\;\cx\in\bR^d,
\end{equation}
with $\mathcal{B}=\{B_{\vec{a}}(\cz,r):\cz\in\bR^d, r>0\}$, the one-parameter family of anisotropic balls. For $\theta>0$ and $|f|^{\theta}\in L_{1,\text{loc}}(\bR^d)$, we define $\cM_\theta f:=\big(\cM |f|^\theta\big)^{1/\theta}$.
We shall need the Fefferman-Stein vector-valued maximal inequality, see \cite{MR2401510}:
If ${p}\in(0,\infty)^d,\ 0<q\leq \infty$, and $0<\theta<\min(p,q)$ then
 for any sequence $\{f_j\}_j$ with $|f_j|^\theta\in L_{1,\text{loc}}(\bR^d)$, we have
\begin{equation}
  \label{eq:fs}
  \|\{(\cM_\theta(f_j)\}\|_{L_{{p}}(\ell_q)}\leq C_B\|\{f_j\}\|_{L_{{p}}(\ell_q)},
\end{equation}
where $C_B:=C_B(\theta,\vec{p},q)$ and the $L_p(\ell_q)$-norm as defined in Eq.\ \eqref{eq:lplq}.

 Let $u(x)$ be a continuous function on $\bR^d$. We define, for $a,R>0$,
$$u^*(a,R;\cx):=\sup_{\cy\in\bR^d}
\frac{|u(\cy)|}{\brac{R^{\vec{a}}(\cx-\cy)}^a},\qquad \cx\in\bR^d.$$
Using a slight variation on standard techniques, see e.g.\ \cite[Proposition 3.3]{Borup2008}, we obtain the following variation on Peetre’s maximal estimate: For every $\theta>0,$ and $c_s>0$,  there exists a constant $c=c(\theta,c_s)>0,$ such that for every $t>0$, $c_f\in\bR^d$, and $f$ with $\supp(\hat{f})\subset t^{\vec{a}}(c_s[-1,1]^d)+c_f,$
\begin{equation}\label{M3}
f^*(\nu/\theta,t;\cx)=\sup_{\cy\in\mathbb{R}^d} \dfrac{|f(\cy)|}{\brac{t^{\vec{a}}(\cx-\cy)}^{\nu/\theta}}\leq c\cM_\theta f(\cx),\;\cx\in\mathbb{R}^d.
\end{equation}   
Theorem \ref{th:vecmult} below provides a vector-valued multiplier result, which combines Peetre's maximal estimate with the Fefferman-Stein vector-valued inequality, and it is used in the proof of Proposition \ref{prop:normchar}.
For $\Omega=\{\Omega_n\}$ a sequence of compact
subsets of $\bR^d$, we let
$$L_{p}^\Omega(\ell_q):=\{\{f_n\}_{n\in\bN}\in
L_{p}(\ell_q)\,|\,\supp(\hat{f}_n)\subseteq\Omega_n,\,\forall n\}.$$ 

 For $s\in (0,\infty)$, let
\begin{equation}\label{eq:sob}
\|f\|_{H_2^s}:=\biggl( \int |\mathcal{F}^{-1}f(x)|^2 \brac{x}^{2s}
dx\biggr)^{1/2}
\end{equation}
denote the Sobolev  norm. The proof of Theorem \ref{th:vecmult}  can be found in \cite[Theorem 3.5]{Borup2008}. 

\begin{mytheorem}\label{th:vecmult}
  Suppose $0<p<\infty$ and
  $0<q\leq \infty$, and let $\Omega=\{T_k \mathcal{C}\}_{k\in\bN}$ be
  a sequence of compact subsets of $\bR^d$ generated by a family
  $\{T_k=t_k^{\vec{a}} \cdot+\xi_k\}_{k\in\bN}$ of invertible affine
  transformations on $\bR^d$, with $\mathcal{C}$ a fixed compact subset of
  $\bR^d$. Assume $\{\psi_j\}_{j\in \bN}$ is a
  sequence of functions satisfying $\psi_j\in H^s_2$ for some $s>\frac
  \nu2+\frac{\nu}{\min(p,q)}$. Then there exists a constant $C<\infty$
  such that 
$$\|\{\psi_k(D)f_k\}\|_{L_{\vec{p}}(\ell_q)} \leq
C\sup_j\|\psi_j(T_j\cdot)\|_{H^s_2}\cdot
\|\{f_k\}\|_{L_{\vec{p}}(\ell_q)}$$
for all $\{f_k\}_{k\in \bN} \in L_{\vec{p}}^\Omega(\ell_q)$.
\end{mytheorem}

Finally, we have the following Lemma, which is used in the proof of Proposition \ref{lem:reconstruct}.
\begin{mylemma}\label{lem:maxbound}
  Let $0<r\leq 1$. Using the same notation as in Proposition \ref{lem:reconstruct}, there exists a constant $C$ such that for any sequence $\{s_{R,\cn}\}_{R,\cn}\subset \C$ we have the following estimate
  $$\sum_{\cn\in\bN_0^d} |s_{R,\cn}||w_{R,\cn}|(\cx)\leq C2^{-d/2}|R|^{1/2} \sum_{m=1}^{2^d} M_r\Bigl(\sum_{\cn\in\bN_0^d}
  |s_{R,\cn}| \mathds{1}_{U(R,\cn)}\Bigr)(O_m \cx),\qquad \cx\in\bR^d.$$
\end{mylemma}

\begin{proof}
 From Eq.\ \eqref{eq:G}, we have the estimate, for any $N\in\bN$, 
  \begin{align}
      |w_{R,\cn}(\cx)|&\leq C_N2^{-d/2}|R|^{1/2}\sum_{m=1}^{2^d}\big(1+\big|
\delta_R(\cx-O_m\ce_{R,\cn})\big|_{\va}\big)^{-N}\nonumber\\       
      & =
  C_N2^{-d/2}|R|^{1/2}\sum_{m=1}^{2^d}\big(1+\big|
\delta_R O_m \cx-\pi(\cn+\ca)\big|_{\va}\big)^{-N}, \label{eq:west}
 \end{align} 
 with $C_N$ independent of $R$,
where we have used that $O_m$ is orthogonal with  $O_m^2=\text{Id}_{\bR^d}$, and that $\delta_R$ and $O_m$ commute as they are diagonal matrices.

 Fix $N>\nu/r$, and let $\cz\in\bR^d$. 
  Put $A_0=\{\cn\in \bN_0^d\colon |\delta_R \cz -\pi(\cn+\ca)|_{\vec{a}}\leq
    1\}$, and for $\ell\in \bN$, we let $A_\ell=\{ \cn\in \bN_0^d\colon  2^{\ell-1}<|\delta_R \cz -\pi(\cn+\ca)|_{\vec{a}}\leq 2^\ell\}$. Notice that $\cup_{\cn\in A_\ell}
    U(R,\cn)$ is a bounded set contained in the ball
    ${B}_{\vec{a}}(\cz,c2^{\ell+1}|R|^{-1/\nu})$. Now, since $|U(R,n)|\asymp |R|^{-1}$ uniformly,
    \begin{align*}
      \sum_{\cn\in A_\ell} &|s_{R,\cn}|\big(1+\big| 
\delta_R\cz-\pi(\cn+\ca)\big|_{\va}\big)^{-N} \\&\leq C2^{-\ell N} \sum_{\cn\in
        A_\ell} |s_{R,\cn}|\\
&      \leq C2^{-\ell N} \Bigl( \sum_{\cn\in
        A_\ell} |s_{R,\cn}|^r\Bigr)^{1/r}\\
      &\leq C2^{-\ell N} |R|^{1/r} \biggl( \int \sum_{\cn\in
        A_\ell} |s_{R,\cn}|^r\mathds{1}_{U(R,\cn)}(y)\, dy \biggr)^{1/r}\\
      &\leq CL^{1-r}2^{-\ell N} |R|^{1/r} \biggl( 
\int_{{B}( \cz,c2^{\ell+1}|R|^{-1/\nu})} \Bigl(\sum_{\cn\in
        A_\ell} |s_{R,\cn}|\mathds{1}_{U(R,\cn)}(y)\Bigr)^r \, dy\biggr)^{1/r}\\
      &\leq C'2^{-\ell(N-\nu/r)} \mathcal{M}_r \Bigl(\sum_{\cn\in
       \bN_0^d} |s_{R,\cn}|\mathds{1}_{U(R,\cn)} \Bigr)(\cz),
    \end{align*}
    with $L$ the uniform constant from Eq.\ \eqref{eq:L1}. 
We now perform the summation over $\ell\in \bN_0$, using that $N-\nu/r>0$, to obtain
$$    \sum_{\cn\in \bN_0^d} |s_{R,\cn}|\big(1+\big| 
\delta_R\cz-\pi(\cn+\ca)\big|\big)^{-N} 
\leq  C\mathcal{M}_r \Bigl(\sum_{\cn\in
       \bN_0^2} |s_{R,\cn}|\mathds{1}_{U(R,\cn)} \Bigr)(\cz).$$
We then use the substitutions $\cz=O_m \cx$, $m=1,\ldots,2^d$, to capture all $2^d$ terms on the RHS of \eqref{eq:west}. 
\end{proof}

\end{document}